\documentclass[11pt,a4paper]{article}

\usepackage[top=37mm, bottom=33mm, left=27mm, right=27mm]{geometry}

\usepackage{cancel}

\usepackage[T1]{fontenc}

\usepackage{amssymb, amsfonts, amsthm}
\usepackage{mathtools}
\usepackage{bm}
\usepackage{accents} 
\usepackage{xcolor}
\definecolor{ms-green}{HTML}{009a55}   
\definecolor{ms-blue}{HTML}{006fb8}   
\definecolor{ms-magenta}{HTML}{ec008d}   

\usepackage{titlesec}
\titleformat{\section}[hang]{\normalfont\scshape\centering}{\thesection.}{1em}{}
\titleformat{\subsection}[runin]
  {\normalfont\sffamily\normalsize\bfseries}{\thesubsection.}{0.5em}{\addperiod}
\titleformat{\subsubsection}[runin]
  {\normalfont\normalsize\itshape}{\thesubsubsection.}{0.5em}{\addperiod}
\newcommand{\addperiod}[1]{#1.}

\makeatletter
\def\@maketitle{%
  \newpage
  \null
  \vskip 2em%
  \begin{center}%
    {\normalfont\normalsize\bfseries\MakeUppercase{\@title} \par}%
    \vskip 1.5em%
    {\normalfont\normalsize
      \lineskip .5em%
      \begin{tabular}[t]{c}%
         \@author
      \end{tabular}\par}%
  \end{center}%
  \par
  \vskip 1.5em}
\makeatother

\usepackage{enumerate}

\numberwithin{equation}{section}

\usepackage[font={small}, textfont={sf}, labelfont={bf,sf}, margin=1cm]{caption}
\newtheorem{assumption}{Assumption}
\newtheorem{theorem}{Theorem}[section]
\newtheorem{proposition}{Proposition}[section]

\newtheorem{lemma}[theorem]{Lemma}
\newtheorem{remark}{Remark}[section]

\usepackage{algorithm}
\usepackage{algorithmic}

\usepackage[bookmarksopen=false, bookmarksnumbered=true, colorlinks=true, linkcolor=ms-blue, citecolor = ms-green, urlcolor = ms-magenta ]{hyperref}

\usepackage{orcidlink}

\hypersetup{
   pdftitle={Low-rank solutions to a class of parametrized systems using Riemannian optimization},
   pdfsubject={math.NA},
   pdfauthor={Marco Sutti and Tommaso Vanzan},
   pdfkeywords={keywords}
}

\renewenvironment{abstract}
 {\small
  \begin{center}
  \bfseries \abstractname\vspace{-.5em}\vspace{0pt}
  \end{center}
  \list{}{
    \setlength{\leftmargin}{0.0cm}%
    \setlength{\rightmargin}{\leftmargin}%
  }%
  \item\relax}
 {\endlist}

\let\Xi\varXi
\let\Phi\varPhi
\let\Psi\varPsi
\let\Pi\varPi
\let\Sigma\varSigma

\newcommand{\email}[1]{\protect\href{mailto:#1}{#1}}

\newcommand*{\vertbar}{\rule[-1ex]{0.5pt}{2.5ex}}

\newcommand{\tr}{^{\top}} 

\newcommand{\dx}{\,\mathrm{d}x}

\makeatletter
\newcommand*{\Relbarfill@}{\arrowfill@\Relbar\Relbar\Relbar}
\newcommand*{\xeq}[2][]{\ext@arrow 0055\Relbarfill@{#1}{#2}}
\makeatother

\DeclareMathOperator{\grad}{grad}
\DeclareMathOperator{\Retraction}{R}
\DeclareMathOperator{\Hessian}{Hess}
\DeclareMathOperator{\D}{D}   
\DeclareMathOperator{\vecop}{vec}
\DeclareMathOperator{\rank}{rank}
\DeclareMathOperator{\diag}{diag}

\DeclareMathOperator{\Proj}{Proj}
\DeclareMathOperator{\trace}{Tr}

\newcommand{\kt}{\ast\tr}  
\newcommand{\PTXM}{\Proj_{X}}

\newcommand{\R}{\mathbb{R}}
\newcommand{\Rmn}{\R^{m \times n}}

\newcommand{\C}{\mathbb{C}}

\newcommand{\cM}{{\mathcal M}}
\newcommand{\cMr}{\cM_{r}} 
\newcommand{\Stmr}{\mathrm{St}_{r}^{m}}
\newcommand{\Stnr}{\mathrm{St}_{r}^{n}}

\newcommand{\T}{\mathrm{T}}
\newcommand{\TXMr}{\mathrm{T}_{X}\cMr}

\newcommand{\Utp}{\widetilde{U}_{\mathrm{p}}}
\newcommand{\Vtp}{\widetilde{V}_{\mathrm{p}}}
\newcommand{\Uhp}{\widehat{U}_{\mathrm{p}}}
\newcommand{\Vhp}{\widehat{V}_{\mathrm{p}}}

\newcommand{\cA}{{\mathcal A}}

\newcommand{\cD}{{\mathcal D}}
\newcommand{\cF}{{\mathcal F}}
\newcommand{\cH}{{\mathcal H}}
\newcommand{\cI}{{\mathcal I}}
\newcommand{\cJ}{{\mathcal J}}
\newcommand{\cO}{{\mathcal O}}
\newcommand{\cP}{{\mathcal P}}
\newcommand{\cT}{{\mathcal T}}

\newcommand{\F}{\mathrm{F}}
\newcommand{\xib}{{\bm \xi}}
\newcommand{\xb}{{\bm x}}

\newcommand{\bb}[0]{{\bm b}}

\newcommand{\Ut}[0]{\widetilde{U}}

\newcommand{\Mh}[0]{\widehat{M}}

\newcommand{\CUR}{\text{CUR}}
\newcommand{\rt}{\widetilde{r}}
\newcommand{\Sigmat}[0]{\widetilde{\Sigma}}

\newcommand{\Vt}[0]{\widetilde{V}}
\newcommand{\Mt}[0]{\widetilde{M}}
\newcommand{\barA}[0]{\bar{A}}
\newcommand{\vb}[0]{{\bm v}}
\newcommand{\zb}[0]{{\bm z}}

\newcommand{\qb}[0]{{\bm q}}
\newcommand{\cb}[0]{{\bm c}}
\newcommand{\rhob}[0]{{\bm \rho}}
\newcommand{\unob}[0]{{\bm 1}}

\newcommand{\wb}{{\bm w}}

\newcommand{\sigmat}{\tilde{\sigma}}

\newcommand*{\dt}[1]{%
  \accentset{\mbox{\large\bfseries .}}{#1}}

\begin{document}

\title{Low-rank solutions to a class of parametrized systems using Riemannian optimization}

\author{\MakeUppercase{Marco Sutti}\thanks{Division of Mathematics, Gran Sasso Science Institute, L'Aquila, Italia (\email{marco.sutti@gssi.it}).}\hspace{2mm}\orcidlink{0000-0002-8410-1372} \MakeUppercase{and Tommaso Vanzan}\thanks{Dipartimento di Scienze Matematiche, Politecnico di Torino, Italia (\email{tommaso.vanzan@polito.it}).}\hspace{2mm}\orcidlink{0000-0001-7554-4692}}
%


\maketitle

\begin{abstract}
We propose a computational framework for computing low-rank approximations to the ensemble of solutions of a parametrized system of the form $A(\boldsymbol{\xi}) \boldsymbol{x}(\boldsymbol{\xi})+g(\boldsymbol{x}(\boldsymbol{\xi}))=b(\boldsymbol{\xi})$ for multiple parameter values.
The central idea is to reinterpret the parametrized system as the first-order optimality condition of an optimization problem set over the space of real matrices, which is then minimized over the manifold of fixed-rank matrices. This formulation enables the use of Riemannian optimization techniques, including conjugate gradient and trust-region methods, and covers both linear and nonlinear instances under mild assumptions on the structure of the parametrized system.
We further provide a theoretical analysis establishing conditions under which the solution matrix admits accurate low-rank approximations, extending existing results from linear to nonlinear problems. 
To enhance computational efficiency and robustness, we discuss tailored preconditioning strategies and a rank-compression mechanism to control the rank growth induced by nonlinearities.
Numerical experiments demonstrate that the proposed approach achieves significant computational savings compared to solving each system independently, as well as highlight the potential of Riemannian optimization methods for low-rank approximations in large-scale parametrized nonlinear problems.



\end{abstract}

\bigskip

\section{Introduction}
This paper is concerned with the numerical solution of the parametrized nonlinear system
\begin{equation}\label{eq:nonlinear_model}
	F(\xb(\xib),\xib) \coloneqq A(\xib)\xb(\xib)+g(\xb(\xib))-\bb(\xib)=0,
\end{equation}
where $ F \colon \R^m\times \Gamma\rightarrow \R^m$, the parameter vector $\xib$ belongs to a compact subset $\Gamma$ of $\R^p$, $A(\xib)\in \R^{m\times m}$ is a symmetric positive definite matrix for every $\xib\in \Gamma$, $\bb(\xib)\in \R^{m}$, and $g \colon \R^m\rightarrow \R^m$ is a smooth nonlinear function acting componentwise, $ g(\xb(\xib))=(g_1((\xb(\xib))_1),\dots,g_m((\xb(\xib))_m ) ) $, $\left\{g_j\right\}_{j=1}^m$ being scalar-valued functions.
We are interested in solving \eqref{eq:nonlinear_model} for many different parameter values $\xib_1,\dots,\xib_{n}$. This task is fundamental in modern scientific computing. Examples arise, for instance, in the construction of projection- and interpolation-based surrogates in reduced order modeling \cite{benner2015survey,quarteroni2015reduced,hesthaven2016certified},
the computation of statistics in uncertainty quantification \cite{lord2014introduction,doi:10.1137/100786356}, and the evaluation of functionals and gradients in parametric design \cite{qian2017certified} or in optimization under uncertainty \cite{heinkenschloss2025optimization,pieraccini2025adaptive}. 
Although the nonlinear system can in principle be solved independently for each parameter value --- and thus in an embarrassingly parallel fashion --- the overall computational cost can be prohibitive. This is especially true when both the problem dimension $m$ and the number of parameters $n$ are large.

A vast literature has been devoted to accelerating the solution of sequences of {\em linear} systems. Relevant contributions include techniques that recycle subspaces across different parameter instances \cite{parks2006recycling,kilmer2006recycling,guido2024subspace}, works that update an initial preconditioner \cite{doi:10.1137/20M1331123} or interpolate a set of precomputed preconditioners for new parameter values \cite{zahm2016interpolation,correnty2022preconditioned,doi:10.1137/24M1628591}, projection-based methods \cite{gu2005numerical}, and truncated low-rank or tensor-based methods \cite{KressnerTobler:20112}. 
A closely related research direction in numerical linear algebra pursues the goal of computing a global approximation of the solution map by seeking an approximated solution of the form
\begin{equation}\label{eq:separation_variable_solution}
	\xb (\xib)\approx \sum_{\nu\in\Lambda} \xb_\nu \psi_\nu(\xib),
\end{equation}
where $\Lambda$ is an index set, $\left\{\psi_\nu \right\}_{\nu\in\Lambda}$ is a fixed, a priori chosen basis (e.g., orthogonal polynomials), and the vectors $\left\{\xb_\nu\right\}_{\nu\in\Lambda}$ are the coefficients to be determined. This has attracted much interest since, considering a {\em linear} version of \eqref{eq:nonlinear_model} and assuming an affine structure of $A(\xib)$, the computation of these coefficients reduces to solving a linear system
\begin{equation}\label{eq:linear_system_SGalerkin}
	\left(\sum_{i=0}^p H_{i}\otimes A_{i} \right) \tilde{{\bm x}} ={\bm \gamma},
\end{equation}
for specific matrices $H_i\in \R^{|\Lambda|\times |\Lambda|}$ and $A_i\in \R^{m\times m}$ (see, e.g., \cite[Ch.~9]{lord2014introduction}), where the vector $\tilde{{\bm x}}=(\xb_{1},\dots,\xb_{|\Lambda|})\tr\in \R^{m |\Lambda|}$ collects the coefficients $\left\{\xb_\nu\right\}_{\nu\in\Lambda}$ while ${\bm \gamma}\in \R^{m |\Lambda|}$ is a suitable right-hand side. 
Despite the large size of the matrix in \eqref{eq:linear_system_SGalerkin}, its particular structure can be exploited to design efficient solvers. Stationary iterative methods and preconditioners for Krylov solvers have been proposed in \cite{powell2009block,rosseel2010iterative,ullmann2010kronecker,ullmann2012efficient,kubinova2020block}. In certain cases, $\tilde{\bm x }$ can be approximated by a vector of low tensor rank, motivating the development of low-rank algorithms \cite{khoromskij2011tensor,benner2015low,lee2017preconditioned,KAYA2024115925}. 
An alternative and equivalent perspective recasts \eqref{eq:linear_system_SGalerkin} into the multilinear matrix equation
\begin{equation}\label{eq:multilinearmatrix}
	A_0 X H_0\tr + \dots + A_p X H_p\tr = \Gamma,
\end{equation}
where $X\coloneqq[{\bm x}_1,\dots,{\bm x}_{|\Lambda|}]\in \R^{m\times |\Lambda|}$. Recent advances include matrix-oriented Krylov methods with low-rank truncations \cite{doi:10.1137/25M1723402,palitta2021convergence}, projection methods \cite{powell2017efficient}, and, particularly relevant to this work, are the contribution on Riemannian optimization algorithms, originally for the Lyapunov equation \cite{vandereycken2010riemannian,kressner2016preconditioned} and recently extended to multilinear matrix equations like \eqref{eq:multilinearmatrix}, see \cite{BKR:2025}.

In contrast to the extensive literature surveyed above, the development of efficient algorithms for {\em nonlinear} systems remains much less mature, see \cite{giraldi2019weakly} for a notable contribution proposing a truncated version of Newton's algorithm. 
The reasons are twofold. On the one hand, most of the techniques previously mentioned strongly leverage properties of Krylov methods or the specific structure of the linear systems \eqref{eq:linear_system_SGalerkin} and \eqref{eq:multilinearmatrix}. Hence, their extensions to the nonlinear setting are not evident. On the other hand, nonlinear operations typically prevent computations from being performed in low-rank formats, which is a key property for computational efficiency.
Among the methodologies surveyed, Riemannian optimization algorithms seem to be the most promising for extension to nonlinear settings.

\subsection{Contributions}
The present manuscript makes three main contributions.

Our first contribution consists of a general framework to compute a low-rank approximation to the matrix
$X=[\xb(\xib_1)|\cdots|\xb(\xib_n)]$, whose columns are the solutions of \eqref{eq:nonlinear_model} for the parameter-values $\xib_1,\dots,\xib_n$. The framework handles both linear and nonlinear instances of \eqref{eq:nonlinear_model} in a unified manner, under common assumptions on the matrices $A(\xib)$ and nonlinearity $g$, see Section~\ref{sec:problem_statement}.
The central idea is to interpret \eqref{eq:nonlinear_model} as the first-order optimality condition of the optimization problem,
\begin{equation}\label{eq:nonlinear_functional}
	\min_{X\in \Rmn}\cF(X)\coloneqq\sum_{j=1}^n m_i\left(\frac{1}{2} \, \xb(\xib_j)\tr \! A(\xib_j)\xb(\xib_j)+\unob\tr G(u(\xib_j)) - \xb(\xib_j)\tr \bb(\xib_j)\right),
\end{equation}
where $G$ contains the primitives of $g_i$ componentwise and $\left\{m_i\right\}_{i=1}^n$ are positive weights.
The functional \eqref{eq:nonlinear_functional} is then minimized over the manifold of rank-$r$ matrices using Riemannian optimization methods (see~\protect{\cite{AMS:2008,Edelman:1998,boumal_2023}}), such as Riemannian conjugate gradient and Riemannian trust-region algorithms. As in \cite{giraldi2019weakly}, our method does not deliver a surrogate model, but an approximation of the solution at parameter values $\xib_1,\dots,\xib_n$. However, if these samples correspond to a set of quadrature or interpolation points, then an interpolatory-based surrogate model can be readily constructed.

Since, a priori, it is not evident whether $X$ has a low-rank structure, our second contribution provides a theoretical justification for the low-rank approximability of $X$ under mild assumptions on the nonlinearity $g$ and the right-hand side $\bb$.
Our analysis extends the results presented in \cite[Sec.~2]{KressnerTobler:20112} to the nonlinear setting using the seminal tools developed in \cite{cohen2015approximation}.

Our third contribution consists of the derivation of robust preconditioning strategies, represented by particular choices of the inner product over the space of real matrices, delivering a convergence behavior of Riemannian optimization algorithms that is robust with respect to the sizes $m$ and $n$.
Furthermore, since the nonlinearity typically increases the numerical rank of $X$, we propose and evaluate a rank-compression strategy to reduce the computational cost and to keep the rank growth under control throughout the solution process.

Numerical experiments demonstrate the efficiency of the proposed approach and the overall computational savings with respect to solving each system independently.

\subsection{Outline of the paper}
The remaining part of the paper is organized as follows. Section~\ref{sec:problem_statement} introduces the parametrized nonlinear system and specifies the mathematical assumptions. It further presents our computational framework based on the reformulation of the parametrized nonlinear system as the optimality condition of a suitable optimization problem over the space of real matrices. Section~\ref{sec:low-rank_approximability} analyses the decay of the singular values of the solution matrix, providing the theoretical justification to search for a low-rank approximation.
Section~\ref{sec:geom_fixed_rank_manif} describes the geometry of the manifold of fixed-rank matrices, and briefly recalls Riemannian optimization algorithms. Section~\ref{sec:lowrankcomp} details the computational cost of evaluating functionals, gradients, and Hessian in low-rank formats.
Numerical experiments arising in the context of parametric PDEs are reported in Section~\ref{sec:experiments}.
Finally, we discuss concluding remarks and future directions in Section~\ref{sec:conclusions}.
Further details about lengthy calculations are available in Appendices~\ref{SM:RiemannianHessian} and~\ref{SM:derivative_svd_truncation}.

\subsection{Notation}
Given any two matrices $A,B\in \Rmn$, we denote by $ \langle A,B\rangle\coloneqq\trace(A\tr\! B) $ the standard Frobenius inner product. The symbol $\|\cdot\|$ denotes the Euclidean norm for vectors and the Frobenius norm for matrices. 
For any $n\in\mathbb{N}$, $\{ {\bm e}_j\}_{j=1}^n$ are the canonical vectors of $\R^n$.
The singular values of $A\in \Rmn$ are denoted by $\left\{\sigma_i\right\}_{i=1}^{\min\{m,n\}}$. The Hadamard product is denoted by $A\odot B$ and $A^{\odot k}$ is the Hadamard product of $A$ with itself $k\in\mathbb{N}$ times, also called Hadamard power.    
Furthermore, for any matrix $X\in \Rmn$ with columns $\xb_1,\dots,\xb_n$, we use the symbol $\vecop(\cdot)$ for the column-stacking vectorization operator, i.e.,
\[
   \vecop(X) = \begin{pmatrix}
      \xb_{1} \\
         \vdots \\
      \xb_{n}
   \end{pmatrix} \in \R^{mn}.
\]

\section{Problem formulation}\label{sec:problem_statement}
In this section, we set up the mathematical framework and relate the solution of a parametrized nonlinear system to the minimization of a suitable functional defined over the set of real matrices. We then compute explicitly the gradient and Hessian of this functional, which will be later needed to formulate Riemannian optimization algorithms. 

The goal of this paper is to compute a low-rank approximation to the ensemble of solutions $\left\{\xb^\star(\xib_j)\right\}_{j=1}^n$ of the parametrized nonlinear system
\begin{equation}\label{eq:nonlinear_model_secPB}
	F(\xb(\xib_j),\xib_j) \coloneqq A(\xib_j)\xb(\xib_j)+g(\xb(\xib_j))-\bb(\xib_j)=0,\quad j=1,\dots,n,
\end{equation}
where $ g \colon \R^m\rightarrow \R^m $ is a smooth nonlinear function acting componentwise,
\[
	g(\xb) = \big( g_1((\xb)_1), \dots, g_m((\xb)_m) \big),
\]
and, for every $\xib=(\xi_1,\dots,\xi_p)$ belonging to a compact subset $\Gamma$ of $\R^p$, $p\in \mathbb{N}$, $A(\xib)\in \R^{m\times m}$ and $b(\xib)\in \R^m $. We further make two additional assumptions. The first one concerns the matrices $A(\xib)$.
\begin{assumption}[Affine representation]\label{ass:affine_structure}
For every $\xib\in\Gamma$, $A(\xib)$ admits the affine parametrization
\[A(\xib)=\barA_0+\sum_{k=1}^p \xi_k \barA_k,\]
with matrices $\barA_k\in\R^{m\times m}$, $k=1,\dots,p$.
\end{assumption}
Assumption \ref{ass:affine_structure} on the affine dependence of $A$ on the parameter vector $ \xib $ is standard in both reduced order modeling (see, e.g., \cite{quarteroni2015reduced}) and in random PDEs, where it is typically justified by a Karhunen--Loève expansion of a random field \cite{lord2014introduction}. In our setting, it will be the key to perform operations in a low-rank format.

The second assumption is concerned with the well-posedness of~\eqref{eq:nonlinear_model_secPB}.
There are different frameworks to analyse the well-posedness of a nonlinear system. A sufficient condition is the Minty--Browder theorem \cite[Sec.~9.14]{Ciarlet}, which guarantees that \eqref{eq:nonlinear_model_secPB} admits a unique solution provided that $F$ is a continuous, coercive and strictly monotone function, i.e., for every $\xib\in \Gamma$ and for any two vectors $\xb,\vb\in \R^m$,
\begin{equation}\label{eq:monotonicity}
	(\xb-\vb)\tr\! \left( F(\xb,\xib)-F(\vb,\xib) \right) > 0, \quad\text{and}\quad \lim_{\|\vb\|\rightarrow \infty} \frac{F(\vb)\tr \vb}{\|\vb\|}=\infty.
\end{equation}
Due to the structure of $F$ in \eqref{eq:nonlinear_model_secPB}, \eqref{eq:monotonicity} holds true if, e.g., $A(\xib)$ is symmetric positive definite for every $\xib$, and $g$ is monotone and coercive. We therefore make the following assumption to guarantee existence and uniqueness of the solution of \eqref{eq:nonlinear_model_secPB} for every $\xib\in\Gamma$.
\begin{assumption}[Well-posedness]\label{ass:well-posedness}
For every $\xib\in\Gamma$, $A(\xib)$ is a symmetric and positive definite matrix and $g$ is
a monotone and coercive function.
\end{assumption}

Our computational framework relies on the interpretation of \eqref{eq:nonlinear_model_secPB} as the first-order optimality condition of the optimization problem
\begin{equation}\label{eq:optimizationproblem}
	\min_{\xb_1,\dots,\xb_n\in\R^m} \widetilde{\cF}(\xb_1,\dots,\xb_n) \coloneqq \sum_{j=1}^n m_j \left(\frac{1}{2} \, \xb_j\tr \! A(\xib_j)\xb_j +\unob_m\tr G(\xb_j)-\xb_j\tr \bb(\xib_j)\right),
\end{equation}
where $G$ is a nonlinear function whose $i$-th component is a primitive of $g_i$, $i=1,\dots,m$, $\unob_m=(1,\dots,1)\tr\in\R^m$, and
$\left\{m_j\right\}_{j=1}^n$ are positive weights. Whenever the parameter vector $\xib$ is randomly distributed and the $\{\xib_j\}_{j=1}^n$ are sampled accordingly, it is natural to impose $\sum_{j=1}^n m_j=1$, so that the empirical average in \eqref{eq:optimizationproblem} is an approximation of the expectation over the probability distribution of $\xib$.
A direct calculation taking variations of $\widetilde{\cF}$ with respect to a perturbation $\delta \vb_j$, $j=1,\dots,n$, shows that the unique minimizer $(\xb^\star_1,\dots,\xb^\star_n)$ of $\widetilde{\cF}$ satisfies precisely \eqref{eq:nonlinear_model_secPB}, and thus $\xb^\star_j=\xb^\star(\xib_j)$ for $j=1,\dots,n$.

Next, we reformulate \eqref{eq:optimizationproblem} as an equivalent minimization problem over the space of $\Rmn$ matrices. To this end, let $X\in \Rmn$ and $B\in\Rmn$ be the matrices 
\[
   X =
  \left[
  \begin{array}{cccc}
    \vertbar & \vertbar &   & \vertbar \\[6pt]
       \xb_{1} &    \xb_{2} &  \cdots  &   \xb_{n}  \\
    \vertbar & \vertbar &   & \vertbar 
  \end{array}
  \right], \qquad
  B =
  \left[
  \begin{array}{cccc}
    \vertbar & \vertbar &   & \vertbar \\[6pt]
       \bb(\xib_1) &    \bb(\xib_2) &  \cdots  &   \bb(\xib_n)  \\
    \vertbar & \vertbar &   & \vertbar 
  \end{array}
  \right].
\]
Using the $\vecop(\cdot)$ operator, the first term of \eqref{eq:optimizationproblem} can be written as
\begin{align*}
  \sum_{j=1}^{n} m_j\left( \frac{1}{2} \, \xb_j \tr\! A(\xib_j) \xb_j\right) = \frac{1}{2}\vecop(X)\tr
   \begin{bmatrix}
 m_1 A(\xib_1)             \\
         &   \ddots              \\
         &          &    m_{n}A(\xib_n)
   \end{bmatrix}
   \vecop(X).
\end{align*}
Furthermore, due to Assumption \ref{ass:affine_structure}, we can equivalently write
\[\begin{bmatrix}
 m_1 A(\xib_1) &           \\
         &   \ddots              \\
         &          &   m_{n}A(\xib_n)
   \end{bmatrix}=\sum_{i=0}^{p} \Xi_{i} \otimes \barA_{i},\]
where $ \Xi_{i}\in \R^{n\times n}$ are the diagonal matrices
\[
   \Xi_{0} \coloneqq
   \begin{bmatrix}
   m_1                 \\
         &      \ddots              \\
         &          &    m_{n}
   \end{bmatrix}
   \quad
   \text{and}
   \quad
   \Xi_{i} \coloneqq
   \begin{bmatrix}
   m_1(\xib_1)_i             \\
         &   \ddots              \\
         &     &            m_{n} (\xib_n)_i
   \end{bmatrix},
\]
for $i\in\left\{1,\dots,p\right\}$. 
Hence, the first term of the discretized functional becomes
\begin{equation}\label{eq:first_term_discr_fun}
   \sum_{j=1}^{n} m_{j}\left( \frac{1}{2} \, \xb_j\tr \! A(\xib_j) \xb_j\right) = \frac{1}{2} \vecop(X)\tr \left( \sum_{i=0}^{p} \Xi_{i} \otimes \bar A_{i} \right) \vecop(X).
\end{equation}
By using the matrix relations (see, e.g., \cite[Ch.~10]{petersen2008matrix})
\begin{equation}\label{eq:trick_kronecker}
 \vecop(AXB) = (B\tr \! \otimes A) \vecop(X), \qquad \vecop(C)\tr \! \vecop(D) = \trace\!\big(C\tr\! D \big),
 \end{equation}
which hold for any matrices $A\in \Rmn$, $X\in\R^{n\times t}$, $B\in \R^{t\times s}$ and $C, D\in \R^{m\times n}$, and the symmetry of $\Xi_{i}$, $ i = 0,\dots,p $, we finally derive the equality
\begin{equation}\label{eq:first_term_discr_fun2}
   \frac{1}{2}\sum_{j=1}^{n} m_{j}\left( \, \xb_j\tr \! A(\xib_j) \xb_j\right) = \frac{1}{2} \sum_{i=0}^{p} \trace \! \left( X\tr\! \barA_{i}X\Xi_{i} \right).
\end{equation}
Using the same algebraic relations, the linear term of the discretized functional~\eqref{eq:optimizationproblem} can be written as
\begin{align*}
   -\sum_{j=1}^{n} m_j \xb_j\tr \bb(\xib_j) &= - \left( \xb_1\tr, \ldots, \xb_{n}\tr \right)
   \begin{bmatrix}
       m_1 I_m   \\
         &     \ddots  \\
         &          &   m_{n} I_m
   \end{bmatrix}
   \begin{pmatrix}
      \bb(\xib_1) \\
         \vdots \\
      \bb(\xib_n)
   \end{pmatrix} \\
   &= -\vecop(X)\tr (\Xi_{0} \otimes I_m) \vecop(B) \\
   &= - \trace(X\tr\! B \Xi_{0}),
\end{align*}
$I_m$ being the identity matrix on $\R^m$.
Finally, the nonlinear term is equivalent to
\[
	\sum_{j=1}^{n} m_j G(\xb_j)\tr \unob_m = \vecop(G(X))\tr (\Xi_{0} \otimes I_{m}) \vecop(\unob_m\unob_n\tr) = \trace(G(X)\tr L),
\]
where we set $ L \coloneqq \unob_m\unob_n\tr \Xi_{0} $ and we still denote with $G$ the nonlinear function that acts columnwise on $X$.
Combining all contributions, the functional $\cF$ is defined over the space $\Rmn$ as
\begin{equation}\label{eq:discr_fun_matrix_form}
   \cF(X) \coloneqq \frac{1}{2} \sum_{i=0}^{p} \trace \! \left( X\tr\! \barA_{i}X\Xi_{i} \right) -  \trace \! \big( X\tr\! B \Xi_{0} \big) + \trace \! \big(G(X)\tr\! L \big),
\end{equation}
which is equal to $\widetilde{\cF}(\xb_1,\dots,\xb_n)$ in~\eqref{eq:optimizationproblem}, provided that $ X = [ \xb_1, \dots, \xb_n ] $.

Our methodology involves finding a low-rank solution through the minimization problem
\begin{equation}\label{eq:optproblem_lowrank}
	\min_{X_r\in \cMr} \cF(X_r),
\end{equation}
$ \cMr \coloneqq \left\{X\in \Rmn \colon \rank(X) = r \right\} $ being the set of matrices of fixed rank $ r \in \mathbb{N} $.
In Section \ref{sec:low-rank_approximability}, under additional conditions on the nonlinearity $ g $, we prove a fast decay of the singular values of the solution matrix $ X $, which justifies the search of a low-rank approximation.

We wrap up this section by concisely deriving both the Euclidean gradient and Hessian of $\cF$, which will be used in the Riemannian optimization algorithms to solve \eqref{eq:optproblem_lowrank}. Recall that the Euclidean gradient of $\cF$ is, by definition, the Riesz representative in $\Rmn$ of the directional derivative $ \D \! \cF\in\mathcal{L}(\Rmn,\R)$ with respect to a chosen inner product. It is well known that the choice of the inner product can significantly affect the convergence of iterative schemes; see, e.g., the seminal works on operator preconditioning for PDEs discretizations \cite{hiptmair2006operator,mardal2011preconditioning,kirby2010functional}, and \cite{NW:2006,doi:10.1137/140970860} for optimization algorithms. Here, we introduce the modified inner product over $\Rmn$ as
\begin{equation}\label{eq:scalar_product}
	\langle A,C \rangle_{\cP} \coloneqq \langle K A \Xi_0, C \rangle=\trace(\Xi_0 A\tr \! K C),
\end{equation}
where $ K \in \R^{m\times m} $ is a symmetric positive definite matrix which, for effective preconditioning, should be spectrally equivalent to the matrices $A(\xib_j)$. Note that $\Xi_0$ is a diagonal matrix with strictly positive diagonal entries, hence it is also symmetric positive definite. 
A direct calculation shows that the directional derivative of the nonlinear term at $X$ along a direction $ \delta X \in \Rmn $ is
\begin{equation}\label{eq:derivative_nonlinear}
   \begin{aligned}
      \D\!\trace \! \left(G(X)\tr \! L\right)[\delta X] &= \trace\! \left(\left(g(X)\odot \delta X\right)^\top \unob_m \unob_n\tr \Xi_0\right) \\
      &= (\Xi_0\unob_n)\tr \left(\delta X\tr \! \odot g(X)^\top \right) \unob_m \\
      & = \trace\!\left(\delta X\tr g(X) \Xi_0\right),
   \end{aligned}
\end{equation}
where in the last step we used the relation \cite[Lemma~7.5.2]{HornJohnson:2013}
\begin{equation}\label{eq:hamadard_and_trace}
{\bm a}^\top ( B\odot C) {\bm d}=\trace\! \left (\diag({\bm a})\tr B \diag({\bm d}) C\tr\right),
\end{equation}
which holds true for any real vectors ${\bm a}\in\R^m$, ${\bm d}\in \R^n$, and real matrices $B$, $C\in\Rmn$.
It follows then that the directional derivative of $\cF$ at $X$ along $ \delta X \in \Rmn $ is
\[
	\D\!\cF(X)[\delta X] = \sum_{i=0}^p \trace(\delta X\tr \! \barA_i X\Xi_i) -\trace(\delta X\tr \! B\Xi_0)+\trace(\delta X\tr \! g(X)\Xi_0),
\]
implying that the gradient of $\cF$ at $X$ expressed with respect to the scalar product \eqref{eq:scalar_product} can be computed by solving the matrix equation
\begin{equation}\label{eq:gradient}
    K \, \bigl(\nabla\cF(X)\bigr) \, \Xi_0 =  \sum_{i=0}^{p} \barA_i X \Xi_i - B\Xi_0 + g(X)\Xi_0.
\end{equation}
This matrix equation can be solved efficiently by precomputing, once and for all, a Cholesky decomposition of $K$, and by directly inverting the diagonal matrix $\Xi_0$. Similarly, the directional derivative of the gradient is computed by taking variations of $ \nabla\cF(X) $ along a direction $\delta X$. A direct calculation leads to the matrix equation
\begin{equation}\label{eq:directional_derivative_gradient}
	K \, \bigl( \D \! \nabla \cF(X)[\delta X] \bigr) \, \Xi_0 = \sum_{i=0}^{p} \barA_i \delta X \Xi_i + (g'(X)\odot \delta X) \Xi_0.
\end{equation}

\section{Low-rank approximability}\label{sec:low-rank_approximability}
In this section, we argue that, under additional regularity assumptions on $g$ and $\bb$, the (possibly full-rank) solution matrix $X$ is well approximated by a low-rank matrix. To this end, we show that the solution map $\xib\mapsto \xb(\xib)$ admits a holomorphic extension to a tensorized Bernstein polyellipse in the complex plane. This in turn permits to derive sharp bounds on the decay of the coefficients of $\xb(\xib)$ expressed in a polynomial basis, leading finally to an estimate of the decay of the singular values of $X$.

We start by proving that the solution map $ \xb \colon \Gamma\ni \xib \mapsto \xb(\xib)\in \R^{m} $ admits a holomorphic extension to an open set $\cO\subset \C^p$ that contains $\Gamma$. Without loss of generality, we assume in this section that $\Gamma=[-1,1]^p$.

\begin{proposition}[Holomorphic extension]\label{proposition:holomorphic_extension}
Let Assumptions \ref{ass:affine_structure} and \ref{ass:well-posedness} hold. Furthermore, assume that $ g_j \colon \R \rightarrow \R $, $j=1,\dots,m$, and $ \bb \colon \R^{p}\rightarrow \R^m $ can be extended to holomorphic maps from $ \C $ to $ \C $ and from $ \C^p$ to $ \C^m$, respectively.  
Then, there exists an open set $\cO\subset \C^p$ containing $\Gamma$, such that the map $\xb \colon \Gamma\rightarrow \R^{m}$ has a holomorphic extension to $\cO$ and satisfies the uniform bound
\begin{equation}\label{eq:uniform_bound}
	\|\xb\|_{L^\infty(\cO,\C^m)}=\sup_{\xib\in\cO}\|\xb(\xib)\|_{\C^m} \eqqcolon C_{\cO}<\infty.
\end{equation}
\end{proposition}
\begin{proof}
The proof consists of verifying the hypotheses of \cite[Thm.~2.5]{cohen2015approximation}, a powerful result based on the holomorphic version of the implicit function theorem in Banach spaces. To this end, we consider the map $F:\C^m\times \C^p \rightarrow \C^m$ defined as
\[
	F(\xb,\xib) \coloneqq A(\xib)\xb + g(\xb) - \bb(\xib).
\]
For every $\xib\in \Gamma$ there exists a unique $\xb(\xib)$ such that $F(\xb(\xib),\xib)=0$, due to Assumption \ref{ass:well-posedness}, and $F$ is holomorphic due to Assumption \ref{ass:affine_structure} and the holomorphicity hypothesis on $g$ and $b$.
In addition, for every $\xib \in \Gamma$,
\[
	\partial_{\xb} F(\xb(\xib),\xib) = A(\xib) + D(\xb(\xib)),
\]
where $ D(\xb(\xib)) = \diag \! \big(g_1'((\xb(\xib))_1),\dots,g_m'((\xb(\xib))_m) \big) $.
Note that $ g_j'(t)\geq 0 $ for every $ t\in \R $ and $j=1,\dots,m$ since $ \{g_j\}_{j=1}^m $ are scalar monotone functions, and thus $ \partial_{\xb} F(\xb(\xib),\xib) $ is an isomorphism from $\R^m$ to $\R^m$ for every $\xib\in\Gamma$.
The claim then follows from \cite[Thm.~2.5]{cohen2015approximation}.
\end{proof}

Next, we show that we can embed into $ \cO $ the Bernstein filled-in polyellipse
\[
	\cH_{\rhob} \coloneqq \otimes_{k=1}^p \cH_{\rho_k}, \quad \text{where}\quad \cH_{\rho_k} \coloneqq \left\{\frac{z+z^{-1}}{2} \colon z\in \C,\; 1\leq |z|\leq \rho_k\right\},
\]
for a suitable choice of the positive real numbers $\{\rho_k\}_{k=1}^p$, with $\rho_k>1$ for every $k=1,\dots,p$. 

\begin{proposition}\label{proposition:BernsteinPolyEllipse}
Let the assumptions of Proposition \ref{proposition:holomorphic_extension} hold.
Then, there exists a $\delta>0$ such that for any sequence $\{\rho_k\}_{k=1}^p$ of real numbers strictly larger than 1 and satisfying
\begin{equation}\label{eq:propositionBernstein}
	\sum_{k=1}^p (\rho_k-1)\|A_k\|\leq \delta,
\end{equation}
the solution map $\xib\mapsto \xb(\xib)$ admits a holomorphic extension over the filled-in ellipse $\cH_{\rhob}$ with the uniform bound \eqref{eq:uniform_bound}.
\end{proposition}
\begin{proof}
The proof is essentially that of \cite[Thm.~2.9]{cohen2015approximation} and here adapted for completeness. Consider the sets 
\[
	\cA(\Gamma)\coloneqq\left\{A\in\R^{m\times m} \colon A=A_0+\sum_{k=1}^p \xi_k A_k,\; \xi_k\in [-1,1]\right\},
\]
and
\[
	\cA(\cO)\coloneqq\left\{A\in\C^{m\times m} \colon A=A_0+\sum_{k=1}^p z_k A_k,\; \zb\coloneqq(z_1,\dots,z_p)\tr \in \cO\right\}.
\]
Due to the compactness of $\cA(\Gamma)$, there exists a $\delta>0$ such that $\cup_{A\in \cA(\Gamma)} B(A,\delta)\subset \cA(\cO)$, $B(A,\delta)$ being the ball of radius $\delta$ centered in $A$. Then, define 
\[
	\cO_{\rho_k}\coloneqq \left\{z\in \C \colon \text{dist}(z,[-1,1]) \coloneqq \min_{y\in [-1,1]} |z-y|\leq \rho_k-1 \right\},
\]
with $\rho_k>1$, $k\in\{1,\dots,p\}$, and note that $\cH_{\rho_k}\subset \cO_{\rho_k}$.
By definition, for every $\zb\in \cO_{\rhob}\coloneqq\otimes_{k=1}^p \cO_{\rho_k}$, there exists a $\xib\in \Gamma$ such that $|\xi_k-z_k|\leq \rho_k-1$. Thus, for every $\zb\in\cO_{\rhob}$, we have
\[
	A(\zb) = A(\xib) + \sum_{k=1}^p (z_k-\xi_k)A_k.
\]
Using \eqref{eq:propositionBernstein}, it then follows that for every $\zb\in\cO_{\rhob}$, $A(\zb)\in \cA(\cO)$ and hence the map $\zb\mapsto \xb(\zb)$ is holomorphic and thus, in particular, also over $\cH_{\rhob}$, concluding the proof.
\end{proof}
We now explicitly construct a low-rank approximant of $X$. To do so, we first consider the multivariate Chebyshev expansion of $\xb$,
\[
	\xb(\xib) = \sum_{\qb\in \mathbb{N}^p} \cb_{\qb}T_{\qb}(\xib) \quad\text{with}\quad T_{\qb}(\xib) \coloneqq T_{q_1}(\xi_1)\cdots T_{q_p}(\xi_p),
\]
$T_{q_k}(\xi)$ being the Chebyshev polynomials of the first kind and 
\[
	\cb_{\qb}\coloneqq\int_{\Gamma} \xb(\xib)T_{\qb}(\xib)\; \mathrm{d}\xib.
\]
The next lemma summarizes well-known results concerning the decay of the Chebyshev coefficients for functions admitting holomorphic extension over $\cH_{\rhob}$.
\begin{lemma}
Let $\xb$ admit a holomorphic extension over the Bernstein polyellipse $ \cH_{\rhob}$ with the uniform bound \eqref{eq:uniform_bound}. Then, the coefficients $\cb_{\qb}\in \R^m$ satisfy
\[\|\cb_{\qb}\|\leq C_{\mathrm{Cheb}} \prod_{k=1}^p \rho_ke^{- q_k},\]
where $ C_{\mathrm{Cheb}} = 2^{\|\qb\|_0/2} C_{\cO}$, with $C_{\cO}$ defined in \eqref{eq:uniform_bound}, and $\|\qb\|_0 \coloneqq |\{j \colon q_j\neq 0\}|$.
\end{lemma}
\begin{proof}
The proof is a direct generalization to the $p$-dimensional case of \cite[Thm.~8.1]{devore1993constructive}; see \cite[Thm.~3.2]{adcock2022sparse} for a complete proof.
\end{proof}

Let $ \cb = (\cb_{\qb})_{\qb\in \mathbb{N}^p}$ be the sequence collecting all Chebyshev coefficients, and for any integer $\ell\in \mathbb{N}$, we denote with $\Lambda_{\ell}\subset\mathbb{N}^p$ the set of the $\ell$ largest ones. The best $\ell$-term approximation of $\xb$ is defined as
\begin{equation}\label{eq:besttermapproximation}
	\widetilde{\xb}(\xib)\coloneqq\sum_{\qb\in \Lambda_{\ell}} c_{\qb} T_{\qb}(\xib),
\end{equation}
and we set
\[
	\widetilde{X}\coloneqq\left[ \widetilde{\xb}(\xib_1),\dots,\widetilde{\xb}(\xib_n)\right]=\begin{pmatrix} \cb_{\qb_1}, &\dots &,\cb_{\qb_\ell}\end{pmatrix}\begin{pmatrix}
T_{\qb_1}(\xib_1) &\cdots & T_{\qb_1}(\xib_n)\\
\vdots & \vdots & \vdots\\
T_{\qb_\ell}(\xib_1) &\cdots & T_{\qb_\ell}(\xib_n)
\end{pmatrix} \in \Rmn,
\]
which is a matrix of rank $\ell$.

\begin{theorem}\label{thm:lowrankapproximability}
Let the assumptions of Proposition \ref{proposition:holomorphic_extension} hold. Then, for any $s\in (0,1]$, there exists a constant $\widetilde{C} = \widetilde{C}(p,s,\rhob) $ such that
\begin{equation}
	\|X-\widetilde{X}\|_{\F} \leq \sqrt{n}\widetilde{C}C_{\cO}\|\cb\|_{s}(\ell+1)^{-\left(\frac{1}{s}-1\right)},
\end{equation}
which implies that the $\ell$-th singular value of $X$ satisfies
\begin{equation}\label{eq:decayellsingularvalue}
	\sigma_{\ell}\leq \sqrt{n}\widetilde{C}C_{\cO}\|\cb\|_{s}(\ell+1)^{-\left(\frac{1}{s}-1\right)}.
\end{equation}
\end{theorem}
\begin{proof}
We start by observing that
\begin{equation}\label{eq:ranklapproximability}
	\|X-\widetilde{X}\|_{\F}^2=\sum_{j=1}^n \|\xb(\xib_j)-\widetilde{\xb}(\xib_j)\|^2\leq \sum_{j=1}^n \left(\sum_{\qb\notin \Lambda_{\ell}} \|\cb_{\qb}\| |T_{\qb}(\xib_j)|\right)^2\leq n\left(\sum_{\qb\notin\Lambda_{\ell}} \|\cb_{\qb}\|\right)^2,
\end{equation}
where we used that $|T_{\qb}(\xib)|\leq 1$ for any $\xib\in \Gamma$. We are thus left to study the best-$\ell$ approximation rates of $\xb$ via multivariate Chebyshev polynomials, for which there is a vast literature of results available. In particular, one can show  (see, e.g., \cite[Lemma 3.5]{adcock2022sparse}) that $\cb\in \ell^s(\mathbb{N}^p)$ for any $0<s<\infty$ and, due to Stechkin's lemma, for any $ 0<s\leq 1 $, it holds that
\[
	\sum_{\qb\notin\Lambda_{\ell}} \|\cb_{\qb}\|\leq \widetilde{C}C_{\cO}\|\cb\|_{s}(\ell+1)^{-\left(\frac{1}{s}-1\right)},
\]
where $\widetilde{C}=\widetilde{C}(p,s,\rhob)$. Substituting into \eqref{eq:ranklapproximability} leads to the first claim.
The second claim follows directly by noticing that $\sigma_\ell\leq \min\limits_{\widehat{X} \colon \rank\{\widehat{X}\}=\ell} \|X-\widehat{X}\|_{\F}\leq \|X-\widetilde{X}\|_{\F} $.
\end{proof}

\begin{remark}
Theorem \ref{thm:lowrankapproximability} shows that the singular values $\{\sigma_\ell\}_{\ell}$ of $X$ decay faster than any polynomial in $\ell$, and justifies our search for a low-rank approximant. We note however that, through much more sophisticated analyses, it is possible to show that the best $\ell$-term approximation error decays with a rate proportional to $\exp(-\ell^{1/p})$; see \cite[Ch.~3]{adcock2022sparse}.
\end{remark}

\section{Geometry of the manifold of fixed-rank matrices and Riemannian optimization} \label{sec:geom_fixed_rank_manif}
In this section, we briefly recall the necessary background to formulate Riemannian optimization algorithms on the manifold of fixed-rank matrices $\cMr$. In particular, due to the choice of inner product \eqref{eq:scalar_product}, we rely on a nonstandard parametrization of the manifold, which in turn requires to derive novel adapted characterizations of a few geometric objects. A similar approach was recently proposed in \cite{BKR:2025}.

It is well-known that $\cMr$ is an embedded submanifold of $\Rmn$; see, e.g., \cite[Sec.~7.5]{boumal_2023}.
Elements in $\cMr$ are typically parametrized using the singular value decomposition (SVD)
\begin{equation}\label{eq:parametrizationmanifold}
    \begin{split}
       \cMr = \lbrace U\Sigma V\tr \colon & U \in \Stmr, \ V \in \Stnr, \\
              & \Sigma = \diag(\sigma_{1}, \sigma_{2}, \ldots, \sigma_{r} ) \in \R^{r \times r}, \ \sigma_{1} \geq \cdots \geq \sigma_{r} > 0 \rbrace,
    \end{split}
\end{equation}
where $ \Stmr $ is the Stiefel manifold of $ m \times r $ real matrices with orthonormal columns, and  $ \Sigma $ is a diagonal matrix with $ \sigma_{1}\geq \sigma_{2}\geq \ldots \geq \sigma_{r}>0 $ on its main diagonal. Note however that the representation of rank-$r$ matrices is not unique.

Given the choice of inner product \eqref{eq:scalar_product} in the ambient space $\Rmn$, we parametrize $\cMr$ via a {\em weighted SVD}~\protect{\cite{VanLoan:1976}}, namely,
\begin{equation}\label{eq:newparametrizationmanifold}
    \begin{split}
       \cMr = \lbrace \Ut\widetilde{\Sigma}\Vt\tr \colon & \Ut\in \R^{m\times r},\; \Vt\in \R^{n\times r},\;\text{s.t. }\Ut\tr\! K \Ut = \Vt\tr\! \Xi_0 \Vt = I_{r}, \\
              & \widetilde{\Sigma} = \diag(\sigmat_{1}, \sigmat_{2}, \ldots, \sigmat_{r} ) \in \R^{r \times r}, \ \sigmat_{1} \geq \cdots \geq \sigmat_{r} > 0 \rbrace,
    \end{split}
\end{equation}
so that any element of $\cMr$ is identified by a triple $(\Ut,\Sigmat,\Vt)$.
We emphasize that the factors $\Ut$ and $\Vt$ satisfy the nonstandard orthogonality conditions $\Ut\tr\! K \Ut = \Vt\tr\! \Xi_0 \Vt = I_{r}$. Hence, this parametrization requires a rederivation of the characterization of the tangent space to $\cMr$ at a point $X$ and of its associated projection.
By explicitly constructing a curve $ c \colon t\mapsto c(t)\in \cMr $ passing through $ X = \Ut\widetilde{\Sigma}\Vt\tr $ at $t=0$, one derives (adapting \cite[Sec.~7.5]{boumal_2023}) the modified tangent space representation 
\begin{equation*}
   \begin{split}
      \mathrm{T}_{X}\cMr = \left\lbrace \Ut\widetilde{M}\Vt\tr + \Utp\Vt\tr + \Ut\Vtp\tr  \right. \colon \widetilde{M} \in \R^{r\times r}, & \ \Utp\in\R^{m\times r}, \ \Vtp \in \R^{n\times r},  \\
      & \left. \Ut\tr\! K \Utp = 0, \ \Vt\tr\! \Xi_0 \Vtp = 0 \right\rbrace.
   \end{split}
\end{equation*}
The last two conditions are the so-called gauge conditions. As Riemannian metric on $\mathrm{T}_{X}\cMr$, we choose the restriction of the Euclidean metric on $\Rmn$ induced by the inner product \eqref{eq:scalar_product}, i.e.,
\begin{equation}\label{eq:riemannian_metric}
	g_{X}(\xi,\xi') = \langle \xi, \xi' \rangle_{\cP} = \trace(K\xi \Xi_0 (\xi')\tr), \quad \text{with} \ \xi, \xi' \in \TXMr.
\end{equation}
For any two tangent vectors $\xi$ and $\xi'$ admitting the representations $ (\Mt,\Utp,\Vtp) $ and $ (\widetilde{M}',\Utp',\Vtp')$, the inner product $\langle \xi,\xi'\rangle_{\cP}$ can be computed efficiently. Indeed, by taking into account the new gauge conditions and the modified orthogonality conditions, a direct calculation leads to 
\begin{align*}
  \langle\xi,\xi'\rangle_{\cP} &= \left\langle K(\Ut\widetilde{M}\Vt\tr + \Utp\Vt\tr + \Ut\Vtp\tr)\Xi_0, \ \Ut\widetilde{M}'\Vt\tr + \Utp'\Vt\tr + \Ut\Vtp'\,\!\tr\right\rangle \\
   &= \trace\!\big(\widetilde{M}\tr\!\widetilde{M}'+\Utp\tr\! K\Utp' + \Utp\tr\! \Xi_0\Utp'\big).
\end{align*}
The tangent space projection $ \PTXM^{\cP} $ maps an arbitrary ambient vector $Z \in \Rmn$ orthogonally, with respect to the $\langle\cdot,\cdot\rangle_{\cP}$ inner product, onto $ \mathrm{T}_{X}\cMr $. Its application is given in abstract terms by \cite[(7.53)]{boumal_2023}
\[
	\PTXM^{\cP}(Z) = P_{\Ut} Z P_{\Vt} + P_{\Ut}^\perp Z P_{\Vt} + P_{\Ut} Z P_{\Vt}^\perp,
\]
where $ P_{\Ut}=\Ut\Ut\tr\! K $ and $ P_{\Vt}=\Xi_0 \Vt \Vt\tr $ are the $K$- and $\Xi_0$- orthogonal projections on the subspaces spanned by the columns of $\Ut$ and rows of $\Vt$, respectively, and $P_{\Ut}^\perp=I-P_{\Ut}$ and $P_{\Vt}^\perp=I-P_{\Vt}$ are the orthogonal projections onto their complements. In terms of the parametrization of $\mathrm{T}_{X}\cMr$, the element $\PTXM^{\cP}(Z)$ is identified by the triple $ (\widetilde{M}, \Utp, \Vtp) $
where 
\[
   \widetilde{M} = \Ut\tr\! K Z\Xi_0 \Vt, \qquad \Utp = Z\Xi_0\Vt - \Ut\widetilde{M}, \qquad \Vtp = Z\tr\! K\Ut - \Vt\widetilde{M}\tr.
\]

The Riemannian gradient of a function $ f \colon \cMr\rightarrow \R $ is \cite[Sec.~3.8]{boumal_2023} 
\[
	\grad f(X)\coloneqq \PTXM^{\cP}\!\left(\nabla \bar{f}(X)\right),
\]
where $ \bar{f} \colon \cO\supset \cMr\rightarrow \R $, satisfying $ \bar{f}|_{\cMr} = f $, is an extension of $f$ defined over an open set $ \cO\subset \Rmn $ containing $\cMr$, $ \nabla\bar{f} $ denotes the Euclidean gradient, and $ \PTXM^\cP $ is the $ \cP $-orthogonal projection onto $\mathrm{T}_{X}\cMr $.
In our setting, the functional $\cF$ defined in \eqref{eq:discr_fun_matrix_form} is naturally defined over the whole $\Rmn$, and its Euclidean gradient is \eqref{eq:gradient}.

Finally, the action of the Riemannian Hessian of $ f \colon \cMr \rightarrow \R $ on a tangent vector $H$ is defined (see \protect{\cite[(7.56)]{boumal_2023}}) as 
\begin{equation}\label{eq:RiemannianHessian}
      \Hessian f(X)[H]  = \PTXM^{\cP}\! \left(\D\!\bar{G}(X)[H]\right),
\end{equation}
where $ \bar{G}(X) $ is a smooth extension of $ \grad f(X) $.
After some lengthy calculations (detailed in~\ref{SM:RiemannianHessian}), one obtains the representation of $\Hessian f(X)[H]$ in the tangent space parametrization
\begin{equation}\label{eq:RiemannianHessianFactors}
	\Hessian f(X)[H] = \Ut\Mh\Vt\tr + \Uhp\Vt\tr + \Ut\Vhp\tr,
\end{equation}
with 
\begin{align}
   \Mh &= \Ut\tr\! K\dt{Z}\Xi_0\Vt, \qquad  \Uhp =  P_{\Ut}^{\perp} ( \dt{Z} \Xi_0\Vt + Z\Xi_0\Vtp\Sigmat^{-1}),\label{eq:hatM}\\
   \Vhp &= (P_{\Vt}^{\perp})\tr ( \dt{Z}\tr\! K\Ut + Z\tr\! K\Utp \Sigmat^{-1} ),\label{eq:hatVp}
\end{align}
where $ Z = \grad \bar{f}(X) $, $ \dt{Z} = \Hessian \bar{f}(X)[H]$, $ X = \Ut\Sigmat\Vt\tr $, and $H$ is represented by the triple $ (\Mt, \Utp, \Vtp) $. 
We emphasize that with respect to the formulas in~\protect{\cite[(7.57)]{boumal_2023}}, here we have the additional matrices $\Xi_0$ and $K$ due to our choice of (preconditioned) Riemannian metric \eqref{eq:riemannian_metric}.

\subsection{Riemannian optimization algorithms}

In the numerical experiments of Sec.~\ref{sec:experiments}, we compute a minimizer of \eqref{eq:optproblem_lowrank} using the Riemannian conjugate gradient (RCG) and the Riemannian trust-region (RTR) methods.

The Riemannian extension of nonlinear conjugate gradient methods originates in \cite{smith1993geometric,smith1994optimization,Edelman:1998,RingWirth:2012}, see \cite{SatoIwai:2015,Sato:2016,Sato:2021,Sato:2022} for a modern treatment.
Beyond standard geometric components (such as projections and retractions) used in Riemannian gradient descent, RCG additionally employs a vector transport to combine search directions across tangent spaces.
RCG generates a sequence of iterates $\{x_i\}_{i\in\mathbb{N}}$ through the iteration
\[
	x_{i+1} = \Retraction_{x_{i}}(\eta_{i}),
\]
where $\eta_i$ is a descent direction and $ \Retraction_{x} \colon \mathrm{T}_{x} \cM \to \cM $ is a retraction, i.e., a smooth approximation of the exponential map satisfying standard first-order consistency conditions (see, e.g., \cite[Def.~4.1.1]{AMS:2008}). The descent direction is typically defined as $ \eta_{i} = \alpha_{i} \, \xi_{i}$,
where the step size $ \alpha_{i} \in\R^+$ is computed using an Armijo backtracking line search applied to the pullback function 
\[
	\widehat{f}_{x_{i}}(\xi) = f(\Retraction_{x_{i}}(\xi)).
\]
and, for $i\geq 1$, 
\[
   \xi_{i} = -\grad f(x_{i}) 
   + \beta_{i}^{\mathrm{FR}} \, \cT_{x_{i-1}\to x_{i}}(\xi_{i-1}),
\]
where $\cT_{x_{i-1}\to x_{i}}\colon \T_{x_{i-1}}\cM \to \T_{x_{i}}\cM$ denotes parallel transport. 
The Fletcher--Reeves parameter is
\begin{equation}\label{eq:fletcher_reeves}
   \beta_{i}^{\mathrm{FR}} 
   = \frac{\|\grad f(x_{i+1})\|^{2}}{\|\grad f(x_{i})\|^{2}}.
\end{equation}
At the first iteration, $\xi_0$ is set equal to the negative Riemannian gradient of $f$ at $x_i$.
A pseudocode description appears in Algorithm~\ref{algo:RCG}. In practice, $\cT_{x_{i-1}\to x_{i}}$ is approximated by the orthogonal projection onto the $\T_{x_{i}}\cM$.

The choice of $\beta_i$ is critical. We adopt the Fletcher--Reeves update~\eqref{eq:fletcher_reeves}~\cite{Fletcher:1964}; alternatives include the Polak--Ribi\`ere~\cite{Polak:1969} and Hestenes--Stiefel~\cite{Hestenes:1952} formulas. 
Convergence results for the Fletcher--Reeves variant are given in~\cite[\S 6.1.1]{Sato:2022} and~\cite[\S 4.4]{Sato:2021}.

\begin{algorithm}
	\caption{The RCG method.} \label{algo:RCG}
	\begin{algorithmic}
		\STATE \textbf{Input}: Riemannian manifold $\cM$, objective function $f$, initial iterate $x_{0}\in \cM$, projector $ \P_{x} $ onto $\T_{x}\cM$, retraction $\Retraction_{x}$ from $\T_{x}\cM$ to $\cM$.
		\STATE \textbf{Output}: Sequence of iterates $\lbrace x_{i} \rbrace$.
		\STATE Set $i\leftarrow 0$ and compute the Riemannian gradient as $ \grad f(x_{i}) = \P_{x_{i}} \! \big( \nabla f(x_{i}) \big)$.
		\STATE Set $ \xi_{i} =  -\grad f(x_{i}) $.
		\WHILE{$x_{i}$ does not sufficiently minimize $f$}
			\STATE Compute a step size $\alpha_{i} > 0$ with a line-search procedure.
			\STATE Set $x_{i+1} = \Retraction_{x_{i}}(\alpha_{i} \xi_{i})$.
			\STATE Compute the Riemannian gradient as $ \grad f(x_{i+1}) = \P_{x_{i+1}}\! \big( \nabla f(x_{i+1}) \big)$.
			\STATE Compute the parallel transport $ \cT_{x_{i}\to x_{i+1}} $ of $ \xi_{i} $.
			\STATE Compute the scalar $\beta_{i}^{\mathrm{FR}} = \frac{\|\grad f(x_{i+1})\|^{2}}{\|\grad f(x_{i})\|^{2}} $.
			\STATE Set the new direction $ \xi_{i+1} = -\grad f(x_{i+1}) + \beta_{i}^{\mathrm{FR}} \, \cT_{x_{i}\to x_{i+1}} \, \xi_{i} $.
			\STATE $ i \leftarrow i + 1 $.
		\ENDWHILE
	\end{algorithmic}
\end{algorithm}

In the numerical experiments, we also employ the Riemannian trust-region (RTR) method of Absil, Baker, and Gallivan~\protect{\cite{ABG:2007}}. For reference, we provide the pseudocode for RTR in Algorithm~\ref{algo:RTR}.
The main idea of trust-region methods is that, instead of minimizing the objective function, at the current iterate we construct a quadratic model of it and minimize it in a region where the model's accuracy is trusted.

Given an initial point $ x_{0} \in \cM $ and an initial trust-region radius $ \Delta_{1} \in (0, \bar{\Delta}) $, the method generates a sequence $ \{ x_{i} \}_{i\in\mathbb{N}} \subset \cM $.
At iteration $ i $, a quadratic model $m_i: \mathrm{T}_{x_{i}}\cM\rightarrow \R$ of $ f $ around $ x_{i} $ is constructed as
\begin{equation}\label{eq:quadratic_model}
	m_{i}(\xi) = f(x_{i}) + \langle \grad f(x_{i}), \xi \rangle + \tfrac{1}{2}\langle \Hessian f(x_{i})[\xi], \xi \rangle, \qquad \xi \in \mathrm{T}_{x_{i}}\cM.
\end{equation}
A trial step $ \eta_{i} $ is obtained by (approximately) solving the trust-region subproblem
\begin{equation}\label{eq:TRsubproblem}
	\min_{\|\xi\| \leq \Delta_{i}} m_{i}(\xi).
\end{equation}
This is typically done using the truncated conjugate gradient (tCG) of~\protect{\cite{Steihaug:1983,Toint:1981}}.
The step acceptance is based on the ratio of achieved to predicted reduction,
\[
	\rho_{i} = \frac{\widehat{f}_{i}(0) - \widehat{f}_{i}(\eta_{i})}{m_{i}(0) - m_{i}(\eta_{i})},
\]
where $ \widehat{f}_{i} $ denotes the pullback of $ f $ to $ \mathrm{T}_{x_{i}}\cM $.

If $ \rho_{i} \geq 0.05 $, the step is accepted and the iterate is updated via the retraction, $ x_{i+1} = \Retraction_{x_{i}}(\eta_{i})$,
otherwise, the step is rejected and $ x_{i+1} = x_{i} $.

The trust-region radius $ \Delta_{i} $ is updated according to standard rules. In particular, the radius is increased when the model is reliable ($ \rho_{i} \geq 0.75 $) and the step lies on the boundary of the trust region; it is reduced when the agreement is poor ($ \rho_{i} \leq 0.25 $); and it is left unchanged otherwise. Under standard assumptions, the RTR method enjoys global convergence guarantees and fast local convergence.

\begin{algorithm}
    \caption{The RTR method of~\protect{\cite{ABG:2007}}.}\label{algo:RTR}
    \begin{algorithmic}
      \STATE \textbf{Input}: $ \bar{\Delta} > 0, \, \Delta_{1} \in (0,\bar{\Delta}) $
        \FOR{$ i = 1,2, \dots $}
         \STATE Define the second-order model $ m_{i} $ as in \eqref{eq:quadratic_model}. 
         \STATE Compute $\eta_i$ by solving the TR subproblem \eqref{eq:TRsubproblem}. 
        \STATE Compute $ \rho_{i} = (\widehat{f}_{i}(0) - \widehat{f}_{i}(\eta_{i}))/(m_{i}(0) - m_{i}(\eta_{i})) $.
        
  \IF{$ \rho_{i} \geq 0.05 $}
 	\STATE   Accept step and set $ x_{i+1} = \Retraction_{x_{i}}(\eta_{i}) $.
  \ELSE
		\STATE Reject step and set $ x_{i+1} = x_{i} $.
  \ENDIF
 \STATE  Set
  \[
     \Delta_{i+1} =
     \begin{cases}
        \min(2\Delta_{i}, \bar{\Delta}) & \text{if} \ \rho_{i} \geq 0.75 \ \text{and} \ \| \eta_{i} \| = \Delta_{i}, \\
        0.25 \, \| \eta_{i} \| & \text{if} \ \rho_{i} \leq 0.25, \\
        \Delta_{i} & \text{otherwise}.
     \end{cases}     
  \]
  \ENDFOR
  \end{algorithmic}
\end{algorithm}

Algorithms \ref{algo:RCG} and \ref{algo:RTR} assume that the rank $r$ is fixed and known a priori. However, in practice the rank is rarely known in advance. For this reason, rank-adaptive optimization strategies have been developed; see, e.g., \cite{UschmajewVandereycken2014,GaoAbsil:2022,BKR:2025}. In this work, we adopt the rank-adaptive strategy described in \cite[Alg. 4.2]{BKR:2025} and recalled in Alg. \ref{algo:rank_adaptive}.
The idea is to solve \eqref{eq:optproblem_lowrank} for a given rank $r$ and to monitor the norm of the (Euclidean) gradient. If this norm exceeds a prescribed tolerance, the problem \eqref{eq:optproblem_lowrank} is solved again with an increased rank $ r_+ = r + r_{\mathrm{up}}$. This procedure is repeated until the gradient norm satisfies the desired stopping criterion.
To construct a warm start for the optimization problem on $\cM_{r_+}$, we add a normal correction to the current iterate $ X_{k} \in \cMr $ \cite{UschmajewVandereycken2014,GaoAbsil:2022}. More precisely, we set
\[
   X_{k+1} = X_{k} + \alpha_{\ast} Y_{\ast},
\]
with
\[
    Y_{\ast} = \P^{\cP}_{\cM_{r_{\mathrm{up}}}} \! \left( \Proj_{X_{k}}^{\perp,\cP}\!\big(-\nabla_{\cP} \cF(X_{k})\big) \right), \qquad \alpha_{\ast} = \frac{\| Y_{\ast} \|^{2}_{\cP}}{\sum_{i=0}^{p} \left\langle \barA_{i} Y_{\ast} \Xi_{i}, Y_{\ast} \right\rangle_{\F}}.
\]
Here, $ \P^{\cP}_{\cM_{r_{\mathrm{up}}}} $ is the $\cP$-metric projection onto the manifold $ \cM_{r_{\mathrm{up}}} $ (i.e., a rank-$r_{\mathrm{up}}$ truncated SVD), $ \Proj_{X_{k}}^{\perp,\cP} $ is the $\cP$-orthogonal projection onto the normal space $ \mathrm{N}_{X_{k}}\cMr $, and $ \nabla_{\cP}\cF(X_{k}) $ is the gradient with the $\cP$-metric, see~\eqref{eq:gradient}. 

The step size $ \alpha_{\ast} $ is chosen as the minimizer, along the direction $ Y_{\ast} $, of the quadratic and linear part of the functional $ \cF $ (excluding the nonlinear term involving $G$), namely,
\[
   \cF_{\text{quad}}(X) \coloneqq \frac{1}{2} \sum_{i=0}^{p} \trace \! \left( X\tr\! \barA_{i}X \Xi_{i} \right) - \trace(X\tr\! B \Xi_{0}),
\]
see \eqref{eq:discr_fun_matrix_form}. Due to the quadratic structure of $\cF_{\text{quad}}$, minimizing $ \cF_{\text{quad}}(X_k + \alpha Y_{\ast}) $ with respect to $\alpha$ yields the closed-form expression above.

\begin{algorithm}
    \caption{Rank-adaptive optimization strategy \cite[Alg. 4.2]{BKR:2025}. }\label{algo:rank_adaptive}
    \begin{algorithmic}
      \STATE \textbf{Input}: Initial rank $r$ and guess $ X_0 \in \cMr $, rank increment $ r_{\text{up}} $, tolerances $\text{tol}>0$ and $\varepsilon>0$.
	  Initialize $\text{res}=\|\nabla_{\cP}\cF(X_0)\|$, $k=0$.      
      \WHILE{$\text{res}>\text{tol}$}
      	\WHILE{fixed-rank optimization does not converge}
         \STATE Perform one iteration of Alg. \ref{algo:RCG} or \ref{algo:RTR} to get $X_{k+1}=U\Sigma V\tr \in \cM_r$. 
	  	 \IF{$ \sigma^2_r/\sum_{i=1}^r \sigma_i^2 < \varepsilon $}
	  	 	\STATE $r\leftarrow r_-\coloneqq\max\{k \colon \sum_{i=k+1}^r \sigma_i^2/\sum_{i=1}^r \sigma_i^2 \geq \varepsilon \} $.
	  	 	\STATE $ X_{k+1}\leftarrow U(:,1:r)\Sigma(1:r,1:r) V(:,1:r)\tr $.
	  	 \ENDIF
	  	 \STATE $k\leftarrow k+1$.
	  \ENDWHILE
	  \STATE $\text{res}=\|\nabla_{\mathcal{P}}\cF(X_k)\|$.
      \IF{$\text{res}>\text{tol}$}
      	 \STATE $X_k\leftarrow X_k+\alpha_{\ast} Y_{\ast} $.
      	 \STATE $r\leftarrow r_+\coloneqq r+r_{\text{up}}$.
      \ENDIF
  \ENDWHILE
  \end{algorithmic}
\end{algorithm}

\section{Low-rank computations and approximations of nonlinear functions}\label{sec:lowrankcomp}
Whenever $X$ is low rank, it is imperative for computational efficiency to evaluate the functional $\cF$, as well as its gradient and gradient's directional derivative, in low-rank format.
This can be achieved using techniques similar to those described in Appendices A and B of~\protect{\cite{Sutti2024}}. Let $X=\Phi\Psi\tr$, with $ \Phi \in \R^{m\times r} $ and $ \Psi \in \R^{n\times r}$, be a low-rank factorization of $X$, and $B=\Phi_B\Psi_B\tr$,  $ \Phi_B \in \R^{m\times r_B} $ and $ \Psi_B \in \R^{n\times r_B}$ be a low-rank factorization of $B$. Consider first the linear-quadratic problem (i.e., $G\equiv 0$). Then, \eqref{eq:discr_fun_matrix_form} simplifies to
\begin{align*}
   \cF_{\text{quad}}(X) &= \frac{1}{2} \sum_{i=0}^{p} \trace \! \left( \Psi\Phi\tr\! \bar A_{i}\Phi\Psi\tr\!\Xi_{i} \right) -  \trace(\Psi\Phi\tr\!  \Phi_B\Psi_B\tr \Xi_{0}).
\end{align*}
To avoid performing the matrix products $ \Phi\Psi\tr $ and $ \Phi_B\Psi_B\tr$, which cost $\cO(mns)$, with $ s \in \{ r, r_{B} \} $, we use the cyclic property of the trace operator to equivalently formulate $\cF_{\text{quad}}(X)$ as
\begin{equation*}
   \cF_{\text{quad}}(X) = \frac{1}{2} \trace \! \left(\sum_{i=0}^{p} \left(\Phi\tr\! \bar A_{i}\Phi\right)\left(\Psi\tr\!\Xi_{i}\Psi \right)\right)-\trace\!\left(\left(\Phi\tr\!\Phi_B\right)\left( \Psi_B\tr \Xi_0\Psi\right)\right),
\end{equation*}
where we use parentheses to highlight those matrix products that we want to be performed first, whose cost is either $\cO(ms^2)$ or $\cO(ns^2)$, with $s\in\{r,r_B\}$ (assuming that $\bar{A}_i$ is sparse for every $i=0,\dots,p$). The complexity of evaluating the functional is thus $\mathcal{O}((p+1)(m+n)r^2+ (m+n)rr_B)$.
Similarly, the gradient derived in \eqref{eq:gradient} admits the factorized expression
\begin{equation}\label{eq:gradient_factorized}
   \begin{split}
      \nabla \cF_{\text{quad}}(X) = & \ K^{-1} \left[ \bar A_{0}\Phi \quad\bar A_{1}\Phi \quad \cdots \quad \bar A_{p}\Phi \quad -\Phi_{B} \right] \\
      & \begin{bmatrix}
       I_{(p+1)r} &                     \\
               &  I_{r_{B}}       \\
      \end{bmatrix}
      \left(\Xi_0^{-1} \left[ \Xi_0\Psi \quad \Xi_{1}\Psi \quad \cdots \quad \Xi_{p}\Psi \quad \Xi_0\Psi_{B} \right]\right)\tr.
   \end{split}
\end{equation}
The left factor has size $m $-by-$ ((p+1)r+r_{B}) $, the core factor is a diagonal matrix of size $ (p+1)r+r_{B} $, and the right factor is of size $ n $-by-$ ((p+1)r+r_{B})$. Assuming that $K$ can be inverted with a linear computational cost in $m$ and since $\Xi_0$ is diagonal, the computation of the three factors costs $\mathcal{O}((p+1)(m+n)r+(m+n)r_B)$.
Finally, the directional derivative of the gradient (see \eqref{eq:directional_derivative_gradient}) along the direction $H$ (factorized in the ambient space as $ H = U_{\eta} \Sigma_{\eta} V_{\eta}\tr $ and of rank $ r_{H} \coloneqq \rank(H) \leq 2r $) has the factorized expression 
\begin{equation*}
   \begin{split}
      \D\!\left(\nabla \cF_{\text{quad}}(X)\right)[H] = & \ K^{-1} \left[ \bar A_{0}U_{\eta} \quad \bar A_{1}U_{\eta} \quad \cdots \quad \bar A_{p}U_{\eta}  \right] \\
      & \begin{bmatrix}
      I_{p+1} \otimes \Sigma_{\eta}
      \end{bmatrix}
      \left(\Xi_0^{-1}\left[ \Xi_0 V_{\eta} \quad \Xi_{1}V_{\eta} \quad \cdots \quad \Xi_{p}V_{\eta} \right]\right)\tr.
   \end{split}
\end{equation*}
The left factor has size $ m \times (p+1) r_{H} $, the core has size $ (p+1) r_{H} \times (p+1) r_{H} $, while the right factor has size $ n \times (p+1) r_{H} $. The cost of the three factors is asymptotically
$\mathcal{O}((p+1)(m+n)r_H)$.

The presence of a nonlinearity makes nontrivial to still perform computations in low-rank format, since it breaks the factorized representation of the input variable $X$. In general, ad-hoc solutions must be tailored to the specific form of the nonlinearity.
In the remainder of this section, we focus on the polynomial nonlinearity $G(X)=\frac{1}{4}W X^{\odot 4}$, $W$ being a diagonal matrix, which will be considered in the numerical experiments of Section \ref{sec:numexp_nonlinear}. We leverage the transposed variant of the Khatri--Rao product (see definition in~\protect{\cite[\S 7]{KressnerTobler:2014}}), denoted by $ \kt $, to compute an {\em exact} factorization of $G(X)$ and of its derivatives. Indeed, let $X$ admit the factorization (not necessarily an SVD) $X=USV\tr\in\Rmn$. Then, $X^{\odot 2}$ is equal to~\protect{\cite[\S 7]{KressnerTobler:2014}}
\begin{equation}\label{eq:fact_Xcirc2}
   X^{\odot 2} = (U \kt U) (S \otimes S) (V \kt V)\tr,
\end{equation}
with $ \rank(X^{\odot 2}) \leq r^{2} $. By introducing the notation $ U_{\odot 2} \coloneqq (U \kt U) $, $ S_{\odot 2} \coloneqq (S \otimes S) $, and $ V_{\odot 2} \coloneqq (V \kt V) $ for the factors of $ X^{\odot 2}$, and by re-applying \eqref{eq:fact_Xcirc2} since $X^{\odot 4} = (X^{\odot 2})^{\odot 2} $, we obtain the exact factorization of $X^{\odot 4} = U_{\odot 4} S_{\odot 4} V_{\odot 4}\tr$ with
\begin{equation*}\label{eq:fact_Xcirc4}
\thickmuskip=0.5mu
\nulldelimiterspace=0.9pt
\scriptspace=0.9pt 
\arraycolsep0.9em 
	U_{\odot 4} \coloneqq (U_{\odot 2} \kt U_{\odot 2}) \in \R^{m \times r^{4}},\;\; S_{\odot 4} \coloneqq (S_{\odot 2} \otimes S_{\odot 2}) \in \R^{r^{4}\times r^4},\;\;   V_{\odot 4} \coloneqq (V_{\odot 2} \kt V_{\odot 2}) \in \R^{n \times r^{4}}. 
\end{equation*}  
Letting $\Phi_{\odot 4} \coloneqq U_{\odot 4}S_{\odot 4}$, $\Psi_{\odot 4} \coloneqq V_{\odot 4}$, and using the definition of $L$ and the trace properties, we can manipulate the nonlinear term in \eqref{eq:discr_fun_matrix_form} obtaining the simplified expression
\begin{equation}\label{eq:functionalnonlinear}
	\trace (G(X)\tr L) = \frac{1}{4}\trace (W \Phi_{\odot 4} \Psi_{\odot 4}\tr \Xi_0 \unob_n \unob_m\tr)= \frac{1}{4}(\wb \tr \Phi_{\odot 4})(\Psi_{\odot 4}\tr{\bm m}), 
\end{equation}
where $\wb=(w_1,\dots,w_m)\in \R^m$ is the vector containing the mass-lumped quadrature weights, and ${\bm m}\coloneqq\diag(\Xi_0)\in \R^n$. Note that the last operation in~\eqref{eq:functionalnonlinear} is a scalar product between two vectors of length $r^4$.
Substituting \eqref{eq:functionalnonlinear} into \eqref{eq:discr_fun_matrix_form} yields the nonlinear functional expressed in low-rank format
\begin{equation*}
      \cF(X) = \trace \! \left( \frac{1}{2} \sum_{i=0}^{p} \left(\Phi\tr\! \bar A_{i}\Phi\right)\left(\Psi\tr\!\Xi_{i}\Psi \right)-\left(\Phi\tr \Phi_B\right)\left( \Psi_B\tr \Xi_0\Psi\right)\right)  + \frac{1}{4}(\wb \tr \Phi_{\odot 4}) (\Psi_{\odot 4}\tr{\bm m}),
\end{equation*}
which compared to the linear-quadratic case leads to an additional cost of order $\mathcal{O}((m+n)r^4)$, also including the cost of the transposed variant of the Khatri--Rao product.
Similarly, $ X^{\odot 3} $ admits the exact factorization, 
\[
   X^{\odot 3} = \left( U \kt U_{\odot 2} \right) \left( S \otimes S_{\odot 2} \right) \left( V \kt V_{\odot 2} \right)\tr,
\]
and, denoting the three terms by $ U_{\odot 3} $, $ S_{\odot 3} $, and $ V_{\odot 3} $ ,
the factorized expression of the Euclidean gradient is
\begin{equation*}
   \begin{split}
      \nabla \cF(X) = & \ K^{-1}\left[ \bar A_{0}\Phi \quad \bar A_{1}\Phi \quad \cdots \quad \bar A_{p}\Phi \quad -\Phi_{B} \quad W \, U_{\odot 3} \right] \\
      & \begin{bmatrix}
       I_{(p+1)r} &             &      \\
               &  I_{r_{B}}  &      \\
               &             &  S_{\odot 3}
      \end{bmatrix}
     \left(\Xi_0^{-1} \left[ \Psi \quad \Xi_{1}\Psi \quad \cdots \quad \Xi_{p}\Psi \quad \Xi_0\Psi_{B} \quad  \Xi_0 V_{\odot 3} \right]\right)\tr .
   \end{split}
\end{equation*}
Compared to \eqref{eq:gradient_factorized}, the rank of the Euclidean gradient increases by a factor $r^3$ and has an additional cost of order $\mathcal{O}((m+n)r^3)$.
Concerning the directional derivative of the Euclidean gradient, the nonlinear term leads to the additional contribution
\[
   3\left( W X^{\odot 2} \odot H \right) \Xi_{0}=3W\left( X^{\odot 2} \odot H \right) \Xi_{0},
\] 
where $ H=U_\eta S_\eta V_\eta\tr \in \mathrm{T}_{X}\cMr $. The exact factorization of $ X^{\odot 2} \odot H $ can be once more derived via the transposed variant of the Khatri--Rao product, namely
\[
   X^{\odot 2} \odot H = U_{\odot} S_{\odot} V_{\odot}\tr,
\]
where $ U_{\odot} \coloneqq \left( U_{\odot 2} \kt U_{\eta} \right) $, $ S_{\odot} \coloneqq \left( S_{\odot 2} \otimes S_{\eta} \right) $, and $ V_{\odot} \coloneqq \left( V_{\odot 2} \kt V_{\eta} \right) $.
The full expression of the factorized directional derivative of the gradient is
\begin{equation*}
   \begin{split}
      \D\!\nabla \cF(X)[H] = & \ K^{-1} \left[ \bar A_{0}U_{\eta} \quad \bar A_{1}U_{\eta} \quad \cdots \quad \bar A_{p}U_{\eta} \quad W \, U_{\odot} \right] \\
      & \begin{bmatrix}
      I_{(p+1)} \otimes \Sigma_{\eta}   &    \\
                                      &   3S_{\odot}
      \end{bmatrix}
      \left(\Xi_0^{-1}\left[ \Xi_0V_{\eta} \quad \Xi_{1}V_{\eta} \quad \cdots \quad \Xi_{p}V_{\eta} \quad  \Xi_0V_{\odot} \right]\right)\tr,
   \end{split}
\end{equation*}
whose rank grows at most by a factor $2r^3$ compared to the linear-quadratic case and leads to the additional cost $\mathcal{O}((m+n)r^3)$.

\subsection{Alternative handling of nonlinearities}\label{sec:nonlinearcomputations}
While the transposed Khatri--Rao product yields an exact factorization of the quartic nonlinearity, it typically leads to a substantial rank increase which may become prohibitive in certain large-scale simulations.
This motivates the development of approximated strategies based either on rank compression techniques or low-rank matrix factorizations.

At first glance, a promising approach for computing a low-rank approximation of $f(X)$, with $f$ a generic nonlinearity applied componentwise, would be to exploit CUR factorizations (see, e.g., \protect{\cite{Drineas:2008,Mahoney:2009}}). The central idea of the CUR decomposition is to derive a factorization of a given matrix $X$ as $X=CUR$, where $ C \coloneqq X(:,\cJ)$, $ R \coloneqq X(\cI,:) $, and $\cI$, $\cJ$ are suitably selected, distinct, row and column indices of $X$. These index sets can be chosen, e.g., using leverage scores as in~\cite{Mahoney:2009}, or (variants of) the DEIM algorithm \cite{chaturantabut2010nonlinear,gidisu2021hybrid}.
To approximate $f(X)$, one may then exploit the componentwise relation $(f(X))_{ij}=f(X_{ij})$, and set $ f(X) \approx f(C)\widehat{U}f(R) \eqqcolon f^\CUR(X)$, where $\widehat{U}\coloneqq f(X(\cI,\cJ))^{-1} \in\R^{k\times k}$, $k=|\cI|=|\cJ|$. Note that $\widehat{U}$ ensures that $f^\CUR(X)(\cI,\cJ)=f(X)(\cI,\cJ)$, that is, the resulting factorization interpolates exactly $f(X)$ for every $(i,j)\in \cI\times \cJ$. Despite its appeal, this approach presents a few limitations in the optimization context. First, $f^\CUR(X)$ does not guarantee that the resulting functional $\cF$ remains bounded from below. Indeed, although $f^\CUR(X)$ approximates $X^{\odot 4}$, this approximation does not preserve nonnegativity, and $f^\CUR(X)$ may contain negative entries. Second, even if positivity is enforced (e.g., by first using a CUR to construct a rank-$\rt$ approximation of $X^{\odot 2}$, and subsequently forming a Khatri--Rao product with itself to approximate $X^{\odot 4}$), the resulting functional is very ill-conditioned and Riemannian optimization algorithms exhibit stagnation at quite high values of the Riemannian gradient norm. 

Therefore, for the sake of this work, we focus on a rank compression technique. For any $X\in\R^{m\times n}$ with weighted SVD format $X=\sum_{i=1}^r \sigma_{i}u_{i}v_i\tr$, $u_j\tr K u_i=\delta_{ji}$,  $v_j\tr \Xi_0 v_i=\delta_{ji}$, let $\cT_{\rt} (X) \colon \R^{m\times n}\rightarrow \cM_{\rt}$ be the best rank $\rt$-approximation of $X$ with respect to the $\langle\cdot,\cdot\rangle_{\mathcal{P}}$ inner product, that is,
\[\cT_{\rt} (X)=\sum_{i=1}^{\rt}\sigma_{i}u_iv_i\tr.\]
Our rank-compression strategy consists in approximating $G(X)=\frac{1}{4}W X^{\odot 4}$ with $G_{\rt}(X)=\frac{1}{4}W  \left(\cT_{\rt}(X) \right)^{\odot 4}$, and in considering the approximated functional 
\begin{equation}\label{eq:Frtilde}
	\begin{aligned}
		\cF^{\rt} (X) \coloneqq & \, \trace \! \left( \frac{1}{2} \sum_{i=0}^{p} \left(\Phi\tr\! \bar A_{i}\Phi\right)\left(\Psi\tr\!\Xi_{i}\Psi \right)-\left(\Phi\tr \Phi_B\right)\left( \Psi_B\tr \Xi_0\Psi\right)\right)\\
		& \ + \frac{1}{4}\underbrace{\trace\! \left( \left(W\left(\cT_{\rt}(X)\right)^{\odot 4}\right) \tr L\right)}_{\eqqcolon T(X)}.
	\end{aligned}
\end{equation}
Note that, in our context, evaluating $\cT_{\rt}$ is trivial since $X$ is always expressed in its weighted SVD format. Hence, the evaluation of the nonlinear term is dominated by the successive Khatri--Rao products, resulting in an overall cost $\mathcal{O}((m+n)\rt^4)$, with $\rt$ possibly much smaller than $r$. For optimization purposes, we need to derive the directional derivative of $T(X)$, which can be achieved via chain rule. A direct calculation leads to
\begin{equation}
	\begin{aligned}
		\D\! T(X)[\delta X] &=\langle W \D (\cT_{\rt}(X))^{\odot 4} [\delta X],L\rangle \\
		&= \langle \D(\cT_{\rt}(X))^{\odot 4} [\delta X], WL\rangle \\
		&= \langle \cT_{\rt}(X)^{\odot 3} \odot \D (\cT_{\rt}(X)) [\delta X], WL\rangle\\
		&= \langle \D (\cT_{\rt}(X)) [\delta X], \cT_{\rt}(X)^{\odot 3} \odot WL\rangle\\
		&=\langle \delta X, \left(\D (\cT_{\rt}(X))\right)^\ast[\cT_{\rt}(X)^{\odot 3} \odot WL]\rangle.
	\end{aligned}
\end{equation}
Detailed calculations for the adjoint (w.r.t. to the standard Frobenius inner product) of the directional derivative of the weighted rank truncation are reported in \ref{SM:derivative_svd_truncation}. Given a $X\in \cM_{r}$ factorized as 
\[X=[U_Z, U_{Z_{\perp}}]\begin{pmatrix}\Sigma_{\rt}\\ & \Sigma_{r-\rt}\end{pmatrix}[V_Z,V_{Z_{\perp}}]\tr,\] 
the final expression for the adjoint is
\begin{equation}\label{eq:directional_derivative}
	\begin{aligned}
		\left( \D (\cT_{\rt}(X))\right)^\ast[\varOmega] = & \ P_{U_Z}^\top \varOmega + \varOmega P_{V_Z} - P_{U_Z}^\top \varOmega P_{V_Z} \\
		& \ + K U_Z \left[ \Psi\tr \odot \varOmega_{21}\tr + (\Xi\tr\! -I )\odot \varOmega_{12}\right]V_{Z_\perp}\tr \Xi_0 \\
		& \ + K U_{Z_{\perp}} \left[ \Psi\odot \varOmega_{12}\tr + (\Xi-I)\odot \varOmega_{21}\right] V_Z\tr \Xi_0,
	\end{aligned}
\end{equation}
where 
$\varOmega_{12}\coloneqq U_Z\tr \varOmega V_{Z_{\perp}} \in \R^{\rt\times (r-\rt)}$, $\varOmega_{21} \coloneqq U_{Z_{\perp}}\tr \varOmega V_Z \in \R^{(r-\rt)\times \rt}$, and
\begin{equation*}
	\Psi\in \R^{(r-\rt)\times \rt},\quad (\Psi)_{ij} \coloneqq \frac{\sigma_i\sigma_j}{\sigma_i^2-\sigma_j^2}\quad\text{and} \quad \Xi\in \R^{(r-\rt)\times \rt},\quad  (\Xi)_{ij} \coloneqq \frac{\sigma_i^2}{\sigma_i^2-\sigma_j^2}.
\end{equation*}
From \eqref{eq:directional_derivative}, the gradient of $T$ is directly obtained multiplying by the inverse of $K$ and of $\Xi_0$, from the left and from the right, respectively; see \eqref{eq:gradient}.

\section{Numerical experiments} \label{sec:experiments}

We present numerical experiments aimed at illustrating the proposed framework with applications arising from parametric/random PDEs. We begin with a linear case to establish a baseline. Subsequently, we consider a nonlinear problem.
All experiments are performed on a workstation with the software MATLAB R2023b and rely on the Manopt toolbox \cite{boumal2014manopt} for Riemannian optimization algorithms. The data structure and routines to store and manipulate elements in $\cM_r$ based on the modified 
representation \eqref{eq:newparametrizationmanifold} were implemented in-house.

\subsection{Linear problem}\label{sec:linear}
We consider the domain $ \cD \coloneqq (0,1)^{2} $, and the weak formulation of the linear parametrized PDE: find $y(\cdot,\xib)\in H^1_0(\cD)$ such that
\begin{equation}\label{eq:linear_random_PDE}
	\int_{\cD} \kappa(x,\xib)  \nabla y(x,\xib) \cdot \nabla v(x) \dx = \int_{\cD} f(x,\xib) v(x) \dx,\quad\forall v\in H^1_0(\cD),\;
\end{equation}
with $\xib\in \Gamma\coloneqq(-1,1)^p$.
The diffusion coefficient $\kappa(x,\xib)$ admits the affine expansion
\begin{equation}\label{eq:diffusion_coefficient}
\kappa(x,\xib)=1+\sum_{j=1}^p \xi_j\psi_j(x),\quad \text{with} \quad\psi_j(x)=(k_j^2+\ell_j^2)^{-1}\sin(\pi k_j x_1)\sin(\pi\ell_j x_2),
\end{equation}
where $(k_j,\ell_j)_{j\geq 1}$ is an ordering sequence of elements in $\mathbb{N}\times\mathbb{N}$ such that $\{k_j+\ell_j\}_j$ is nondecreasing, while the parameter-dependent force term is given by
\begin{equation}
f(x,\xib)\coloneqq 100\,\exp(-((x_1-\xib_1)^2+(x_2-\xib_2)^2)/2)\cos(2\pi x)\sin(2\pi y).
\end{equation}
After randomly selecting $\{\xib_{i}\}_{i=1}^n$ parameter values and discretizing in space using a $\mathbb{P}^1$ finite element space $V_h\coloneqq\text{span}\left\{\varphi_1,\dots,\varphi_m\right\}$ of dimension $m$, we recover the linear version (i.e., $g\equiv 0)$ of \eqref{eq:nonlinear_model}, where $A(\xib)$ is the stiffness matrix, $b(\xib)$ is a discretization of the force term, and the vector $\xb(\xib)$ collects the degrees of freedom of $y(\cdot,\xib)$ in the finite element basis. As preconditioning matrix $K$, we choose the stiffness matrix associated to the standard $H^1_0(\cD)$ scalar product, while we set the weights $m_i=1/n$, for $i=1,\dots,n$. All solves with $K$ are performed using a precomputed Cholesky factorization.

We start by investigating the effect of the modified inner product \eqref{eq:scalar_product} on the convergence behavior of RCG and RTR as both the spatial resolution and the number of samples increase. Table \ref{tab:iterations} reports the number of iterations obtained using the standard Frobenius inner product and the modified one, denoted by the suffix -p. All methods start from a random initial point drawn on the manifold.
Without preconditioning, the convergence behavior of RCG strongly depends on the refinement level. In contrast, RCG-p exhibits remarkable robustness under refinement, with iteration counts remaining stable as $m$ and $n$ grows, highlighting the crucial role played by the choice of the inner product. This observation is in agreement with the results presented in \cite{BKR:2025} for the minimization of a quadratic functional arising in the solution of matrix equations. For the RTR algorithm, we report in parentheses the number of inner tCG iterations. We observe that both preconditioned and unpreconditioned variants require a number of outer iterations independent of the refinement level, while the inner computational cost per iteration of the unpreconditioned version deteriorates drastically as $m$ and $n$ increase.

\begin{table}[h]
\centering
\begin{tabular}{c c c c c c}
\hline
$m$ & $n$ & It.\ RCG & It.\ RCG-p & It.\ RTR & It.\ RTR-p \\
\hline
225 & 256 & 132  & 47 & 27 (582)  & 31 (147) \\
961 & 512 & 321 & 34 & 26 (904) & 29 (107) \\
3969 & 1024 & 759 & 42 & 27 (2377) & 34 (131) \\
16129 & 2048 & $>$1000 & 48 & 33 (4229) & 36 (160) \\
\hline
\end{tabular}\caption{Number of iterations of RCG and RTR with (-p) and without the modified inner product, for increasing values of $m$ and $n$. Parameters $p=4$ and $r=16$.}\label{tab:iterations}
\end{table}

Next, Table \ref{tab:RCG_RTR} compares the preconditioned RCG and preconditioned RTR in terms of computational times and approximation accuracy as the rank $r$ increases. We use the shorthand notation
\[
   e^{s} \coloneqq \frac{\|X-X^s\|_{\mathcal{P}}}{\|X\|_\mathcal{P}},
\]
$X$ being the full snapshot matrix, $s\in \left\{\star,\text{RCG},\text{RTR}\right\}$, to denote, respectively, the relative best approximation error of the best rank-$r$ truncation, and the errors achieved at convergence by the RCG and RTR algorithms in the norm $\|\cdot\|_{\mathcal{P}}$ induced by the scalar product \eqref{eq:scalar_product}. Since Table \ref{tab:iterations} shows that RTR may need quite a few iterations to enter into a quadratic basin of attraction, we performed a warm-up step consisting of RCG iterations until the norm of the Riemannian gradient was smaller than 1.0e-3, before starting the RTR algorithm. We report the number of iterations of RTR only, while the computational times include the preliminary warm-up step.
The results show that RCG outperforms RTR in terms of computational time, despite requiring more iterations. At the same time, both methods achieve the same approximation errors, which are comparable to those of the best approximation. As a reference for comparison, assembling the full snapshot matrix on the same machine (by solving the linear systems with a $K$-preconditioned gradient descent method with optimal step size) took 646.8 seconds.

\begin{table}[ht]
\centering
\begin{tabular}{r r r r r r r r}
\hline
Rank & RCG-It & RCG-s & $e^{\text{RCG}}$ & RTR-It & RTR-s & $e^{\text{RTR}}$  & $e^{\star}$\\
\hline
8 & 34 & 14.2 & 5.0e-03 & 3 & 20.3 & 5.0e-03 & 4.9e-03\\
16 & 33 & 29.0 & 5.9e-04 & 4 & 44.5 & 5.9e-04 & 5.8e-04\\ 
24 & 29 & 39.6 & 1.1e-04 & 4 & 89.8 & 1.1e-05 & 1.1e-05\\
32 & 36 & 62.5 & 3.5e-05 & 24 & 360.2 & 3.5e-05 & 3.5e-05\\
40 & 27 & 57.2 & 1.2e-05 & 24 & 441.5 & 1.2e-05 & 1.5e-05\\
\hline
\end{tabular}\caption{Comparison of RCG and RTR results for the linear problem: number of iterations, computational time (in seconds), and error in the preconditioned Frobenius norm between the resulting low-rank matrices and the full snapshot matrix $X$, for increasing ranks. Parameters: $m=16129$, $n=5000$, and $p=3$.}\label{tab:RCG_RTR}
\end{table}

Finally, Figure \ref{fig:rank_adaptivity} illustrates the convergence behavior of the rank-adaptive Algorithm \ref{algo:rank_adaptive} with both RCG (top row) and RTR (bottom row) as inner solvers, to reach a tolerance on the relative residual (i.e., the norm of the preconditioned Euclidean gradient) of order 1.0e-6.

\begin{figure}
\centering
\includegraphics[width=0.32\textwidth]{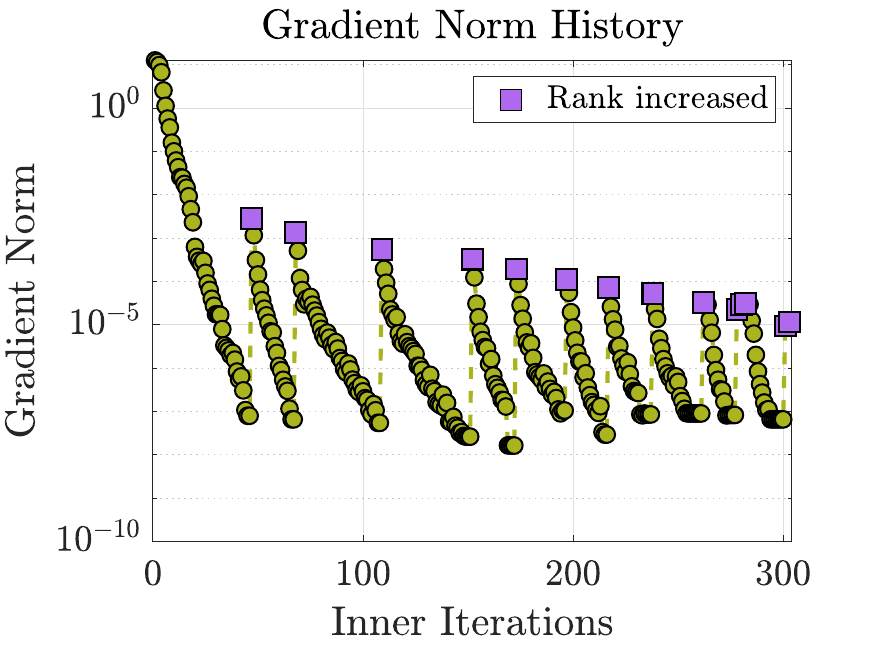}
\includegraphics[width=0.32\textwidth]{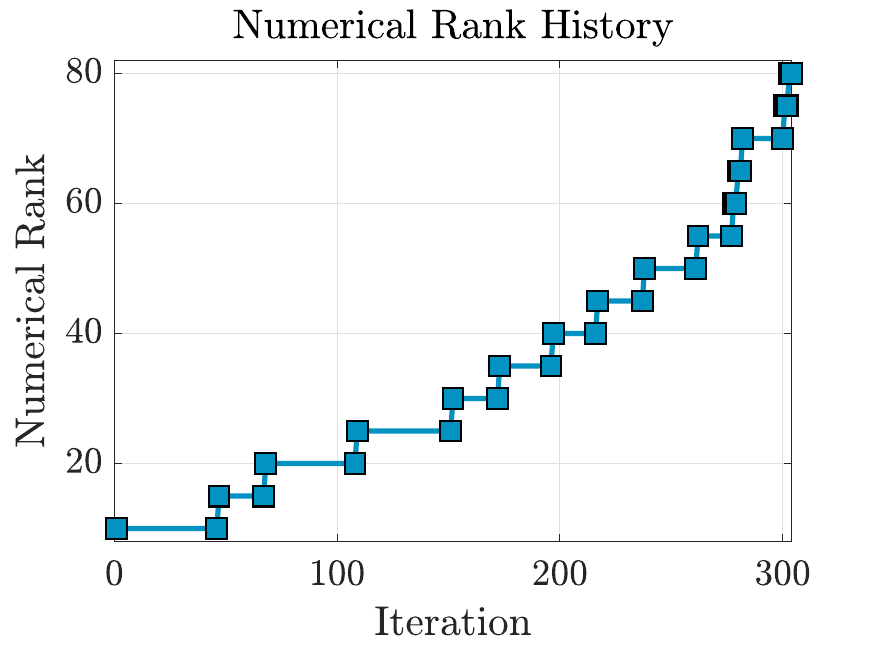}
\includegraphics[width=0.32\textwidth]{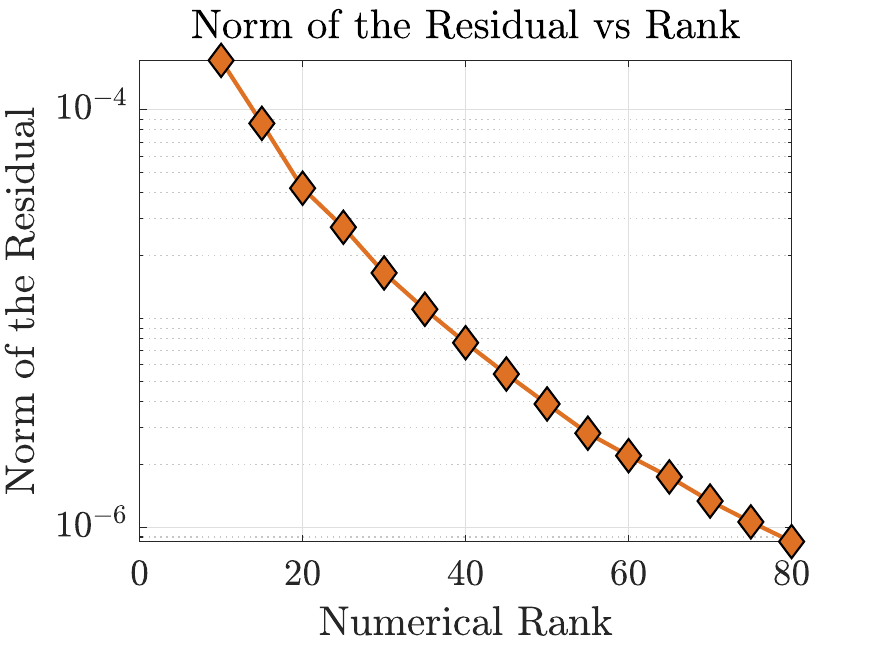}
\includegraphics[width=0.32\textwidth]{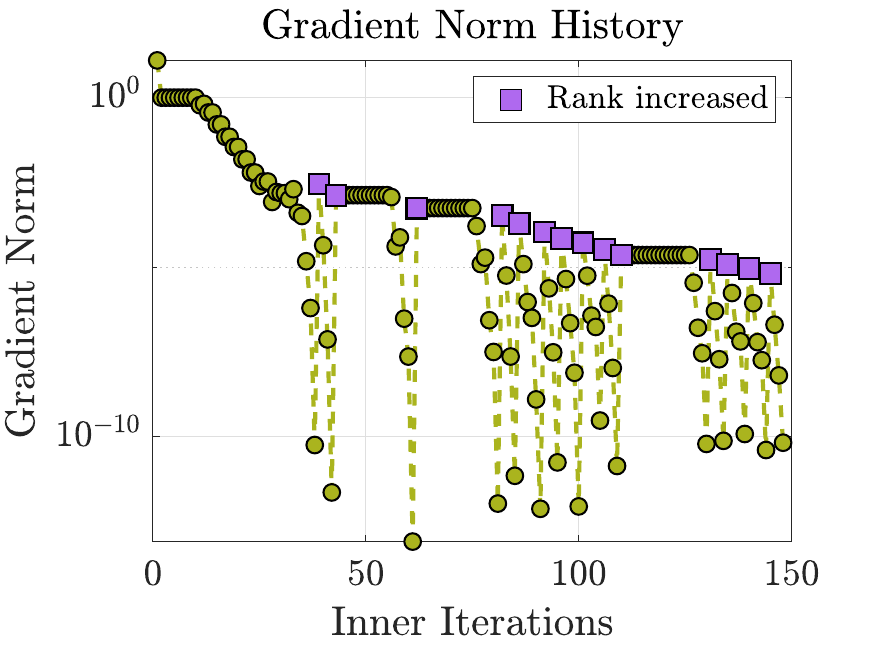}
\includegraphics[width=0.32\textwidth]{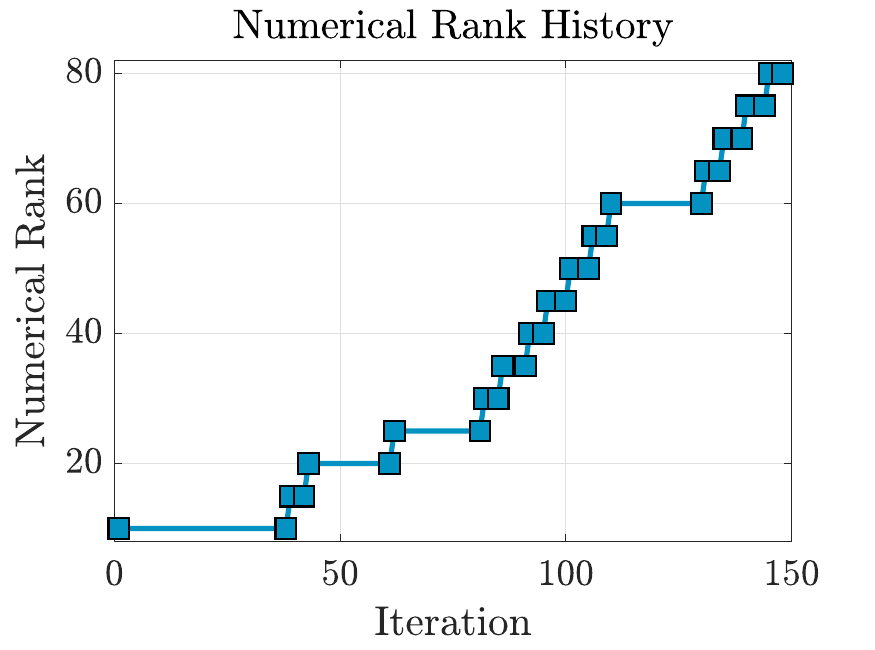}
\includegraphics[width=0.32\textwidth]{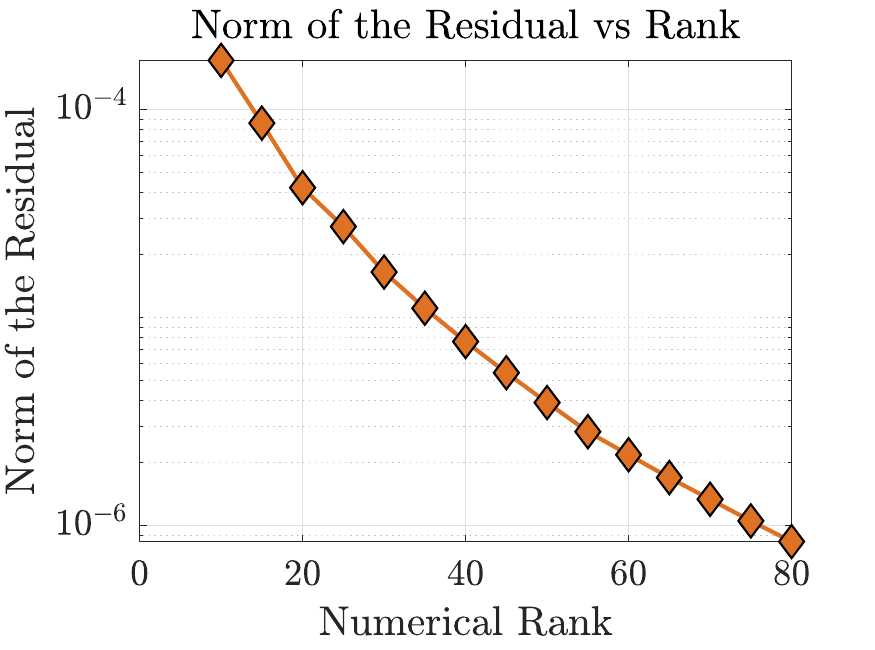}
\caption{Evolution of the norm of the Riemannian gradient, of the residual and of the matrix rank during the rank-adaptive algorithm with RCG (top row) and RTR (bottom row) as inner solvers. Parameters: $m=16129$, $n=5000$, $p=5$, initial rank $5$ and $r_{\mathrm{up}}=5$.}\label{fig:rank_adaptivity}
\end{figure}

\subsection{Nonlinear problem}\label{sec:numexp_nonlinear}

In this section, we consider the weak formulation of the nonlinear parametric PDE: find $y(\cdot,\xib)\in H^1_0(\cD)$ such that
\begin{equation}\label{eq:nonlinear_parametric_PDE}
\int_{\cD} \kappa(x,\xib)  \nabla y(x,\xib) \cdot \nabla v(x) + y^3(x,\xib) v(x) \dx = \int_{\cD} f(x,\xib) v(x) \dx,\quad  \forall v\in H^1_0(\cD),
\end{equation}
with $\xib\in \Gamma \coloneqq [-1,1]^p$, and the diffusion coefficient $\kappa$ is given by \eqref{eq:diffusion_coefficient}.
The discretization of the nonlinear term is performed using a finite element interpolation combined with mass lumping. Specifically, we first approximate the nonlinear term $y\mapsto y^3$ using the finite element interpolation operator $I_h \colon H^1_0(\cD)\rightarrow V_h$. The resulting integral is then approximated using a mass-lumping quadrature rule. For every test function $\varphi_{i}$, this yields
\[\int_{\cD} y^3(x) \varphi_{i}(x)\dx \approx \int_{\cD} I_h (y^3(x)) \varphi_{i}(x)\dx = \sum_{j=1}^m y_j^3 \int_{\cD} \varphi_j(x)\varphi_{i}(x)\dx \approx w_{i} y_{i}^3,\]
where $\left\{w_{i}\right\}_{i=1}^m$ are the positive mass-lumped weights. With this approximation, we recover the nonlinearity discussed in Section \ref{sec:lowrankcomp} and we fit into the abstract framework of Section \ref{sec:problem_statement} by defining $W=\text{diag}(w_1,\dots,w_m)\in\R^{m\times m}$, $g(\xb)=W\xb^{\odot 3}$ and $G(\xb)=\frac{1}{4}W \xb^{\odot 4}$ for every $\xb\in \R^m$, with the convention that both $g$ and $G$ acts columnwise on a matrix $X\in \Rmn$. 

Table \ref{tab:RCG_RTR_nonlinear} reports the number of iterations, computational times and approximation errors of RCG and RTR to reach a tolerance of 1.0e-9 on the Riemannian gradient. As a reference, the assembly of the full snapshot matrix using Newton's method costs 2183 seconds (with an averaged of three iterations to convergence starting from a zero initial guess).
As starting Newton's method with a zero initial guess corresponds to solve in the first iteration the linear problem ($G\equiv 0)$, we started both RCG and RTR with a warmed-up solution computed by solving the linear-quadric optimization problem up to a tolerance of $10^{-3}$ on the Riemannian gradient. This cost is included in the computational times. 
We observe that our computational framework is capable of delivering accurate approximations of the snapshot matrix up to rank $20$, together with a significant saving in computational time. Nevertheless, it is also evident that the $r^4$ scaling of the computational cost represents a severe limitation if higher accuracies are needed.
 
\begin{table}[ht]
\centering
\begin{tabular}{r r r r r r r r r }
\hline
Rank & RCG-It & RCG-s & $e^{\text{RCG}}$ & RTR-It & RTR-s & $e^{\text{RTR}}$  & $e^{\star}$\\
\hline
4 & 18 & 18.3 & 2.2e-02 & 3 & 29.6 & 2.2e-02 & 2.1e-02 \\
8 & 20 & 99.8 & 4.9e-03 & 4 & 336.1 & 4.9e-03 & 4.9e-03  \\ 
12 & 21 & 314.7 & 1.5e-03 & 4 & 1129.0 & 1.5e-03 & 1.5e-03 \\
16 & 21 & 787.2 & 5.9e-04 & 4 & 2268.8 & 5.9e-04 & 5.8e-04 \\
20 & 19 & 1459.7 & 2.7e-04 & 4 & 5155.0 & 2.7e-04 & 2.6e-04\\
24 & 30 & 4405.0 & 1.3e-04 & 24 & 32343.0 & 1.3e-04 & 1.3e-04\\
\hline
\end{tabular}\caption{Comparison of RCG and RTR results for the nonlinear problem: number of iterations, computational time (in seconds), and the error in the preconditioned Frobenius norm between the resulting low-rank matrices and the full snapshot matrix $X$, for increasing ranks. Parameters: $m=16129$, $n=5000$, and $p=3$.}\label{tab:RCG_RTR_nonlinear}
\end{table}

We conclude this section with Table \ref{tab:RCG_nonlinear_Comp}, which compares the computational efficiency of the truncation strategy discussed in Sec.~\ref{sec:nonlinearcomputations} (marked by a `C') to the exact factorization of $X^{\odot 4}$. The truncation parameter $\rt$ is set equal to half of the manifold rank $r$.
Table \ref{tab:RCG_nonlinear_Comp} shows that the compression strategy is quite promising as, e.g., an accuracy of $8.2\cdot 10^{-4}$ is achieved with half the computational time of the approach employing the exact factorization. Nevertheless, the scaling $\rt=r/2$ can be too aggressive as $r$ increases, leading to less accurate approximations.

\begin{table}[ht]
\centering
\begin{tabular}{r r r r r r r r }
\hline
Rank & RCG-It & RCG-s & $e^{\text{RCG}}$ & RCG-C-It & RCG-C-s & $e^{\text{RCG-C}}$  & $e^{\star}$ \\
\hline
6 & 21 & 45.4 & 9.4e-03 & 19  & 22.2 &  9.4e-03 & 9.2e-03\\
10 & 23 & 200.4 & 2.5e-03 & 25  & 79.8 & 2.9e-04 & 2.4e-03\\ 
14 & 27 & 617.1 & 9.9e-04 & 100 & 804.2 & 1.3e-03 & 9.8e-04\\
18 & 41 & 1976.9 & 4.2e-04 & 23  & 395.1 & 8.2e-04 & 4.1e-04\\
22 & 20 & 1864.9 & 1.9e-04 &  26 & 844.5.2 & 7.3e-04 & 1.8e-04\\
\hline
\end{tabular}\caption{Number of iterations and computational times (in seconds) of RCG for the nonlinear functional with exact factorization of the nonlinearity and with the compressed approximation, for increasing ranks. Parameters: $m=16129$, $n=5000$, and $p=3$.}\label{tab:RCG_nonlinear_Comp}
\end{table}

\section{Conclusions} \label{sec:conclusions}

We have presented a computational framework based on Riemannian optimization for computing a low-rank approximation to the ensemble of solutions of parametrized, possibly nonlinear, systems. The approach is complemented by a rigorous analysis of the decay of the singular values of the solution matrix, under the common assumption of holomorphic regularity of the system's coefficients on the parameters.
Numerical experiments demonstrate that the proposed framework is competitive against assembling the full snapshot matrix, even at relatively stringent accuracy tolerances.

Several directions remain open for future investigation. First, the development of effective Hessian approximations could significantly improve the efficiency of trust-region methods, which in our experiments were not competitive against Riemannian conjugate gradient. A natural starting point would be a mean-based approximation, which is a popular strategy for stochastic Galerkin systems; see, e.g., \cite{powell2009block}. 

Second, as this is among the first works to address low-rank representations of nonlinearities, several computational challenges remain open. While a compression strategy has been proposed to mitigate the $r^4$ scaling arising from the quartic nonlinearity, more tailored approaches could further improve efficiency without sacrificing accuracy. One promising idea is to delay rank compression by exploiting the compositional structure of the nonlinearity, for instance via $X\mapsto (\cT_{\tilde{r}}(X^{\odot 2}))^{\odot 2}$. Our experiments suggest this preserves accuracy more reliably, though significant computational gains remain elusive since the cost of the rank-$\tilde{r}$ truncation of $X^{\odot 2}$ still scales as $r^4$.
Randomized truncation techniques may alleviate this bottleneck, provided they remain simple enough to admit closed-form derivatives.

Finally, the treatment of general nonlinearities (i.e., beyond the polynomial case) poses the broader challenge of evaluating nonlinear functions in low-rank format within an optimization algorithm. Extending techniques from reduced order modeling, such as CUR factorizations and DEIM, to settings compatible with Riemannian optimization algorithms is a compelling direction that deserves further exploration.

\section*{Acknowledgments}
MS and TV are members of the INdAM-GNCS group. MS acknowledges a visit to the Politecnico di Torino funded by the National Center for Theoretical Sciences (NCTS) of Taiwan. TV acknowledges a funded visit at the NCTS and has been partially funded by Progetto di Ricerca GNCS-INdAM, CUP\_E53C25002010001.

\appendix

\section{Riemannian Hessian with a preconditioned inner product} \label{SM:RiemannianHessian}
In this appendix, we report the calculations to find the parameterization $ ( \Mh, \Uhp, \Vhp ) $ of the Riemannian Hessian $ \Hessian f(X)[H] $.

Let $ X = \Ut \Sigmat \Vt\tr $.
Following \cite[Ch.~7.5]{boumal_2023}, we define the smooth extension of $\grad f(X)$ as 
\[
	\bar{G}(X) = \Proj_{X}(Z) = P_{\Ut} ZP_{\Vt}+P_{\Ut}^\perp Z P_{\Vt}+P_{\Ut} Z P_{\Vt}^\perp=ZP_{\Vt}+P_{\Ut} Z-P_{\Ut} ZP_{\Vt},\]
where $ Z = \grad \bar{f}(X) $, with $ \bar{f}(X) $ the smooth extension of $ f(X) $. The directional derivative of $ \bar{G}(X) $ along a tangent vector direction $H$ is
\begin{equation}\label{eq:derivativesmoothextensionGradient}
   \D\!\bar{G}(X)[H] = P_{\Ut}\dt{Z}P_{\Vt} + P_{\Ut}^{\perp} (\dt{Z} P_{\Vt} + Z \dt{P}_{\Vt}) + (P_{\Ut} \dt{Z} + \dt{P}_{\Ut} Z) P_{\Vt}^{\perp},
\end{equation}
where $\dt{Z}=\Hessian \bar{f}(X)[H|$.
Here, the projectors are modified according to the (preconditioned) inner product, i.e.,
\[
   P_{\Ut} = \Ut\Ut\tr\! K, \quad \P_{\Ut}^{\perp} = I - \Ut\Ut\tr\! K, \quad P_{\Vt} = \Xi_0\Vt\Vt\tr, \quad P_{\Vt}^{\perp} = I - \Xi_0\Vt\Vt\tr.
\]
Using the same curve as in~\protect{\cite[Ch.~7.5]{boumal_2023}}, one obtains the closed expressions for $\dt{P}_{\Ut}$ and $\dt{P}_{\Vt}$,
\begin{equation}\label{eq:modified_projectors}
   \dt{P}_{\Ut} = \Ut\Sigmat^{-1}\Utp\tr\! K + \Utp \Sigmat^{-1}\Ut\tr\! K, \qquad \dt{P}_{\Vt} = \Xi_0\Vt\Sigmat^{-1}\Vtp\tr + \Xi_0\Vtp \Sigmat^{-1}\Vt\tr,
\end{equation}
where $ ( M, \Utp, \Vtp ) $ is the triple representing $H$. 

Starting from~\eqref{eq:derivativesmoothextensionGradient}, we distribute the products
\begin{align*}
   \D\!\bar{G}(X)[H] &= P_{\Ut}\dt{Z}P_{\Vt} + P_{\Ut}^{\perp} \dt{Z} P_{\Vt} + P_{\Ut}^{\perp} Z \dt{P}_{\Vt} + P_{\Ut} \dt{Z} P_{\Vt}^{\perp} + \dt{P}_{\Ut} Z P_{\Vt}^{\perp} \\
   &= \underbrace{P_{\Ut}\dt{Z}P_{\Vt} + P_{\Ut}^{\perp} \dt{Z} P_{\Vt} + P_{\Ut} \dt{Z} P_{\Vt}^{\perp}}_{\PTXM(\dt{Z})} + P_{\Ut}^{\perp} Z \dt{P}_{\Vt} + \dt{P}_{\Ut} Z P_{\Vt}^{\perp}.
\end{align*}
We recognize, in the first three terms on the right-hand side, the projection of $ \dt{Z} $ onto the tangent space $ \mathrm{T}_{X}\cMr $, defined as
\begin{equation}\label{eq:PTXMdtZ}
   \PTXM(\dt{Z}) = P_{\Ut}\dt{Z}P_{\Vt} + P_{\Ut}^{\perp}\dt{Z}P_{\Vt} + P_{\Ut}\dt{Z}P_{\Vt}^{\perp}.
\end{equation}

Substituting~\eqref{eq:modified_projectors}, we get
\begin{equation}\label{eq:B.4}
   \begin{split}
      \D\!\bar{G}(X)[H] &= \PTXM(\dt{Z}) + P_{\Ut}^{\perp} Z (\Xi_{0} \Vt \Sigmat^{-1}\Vtp\tr + \Xi_0\Vtp \Sigmat^{-1}\Vt\tr) \\
      &+ (\Ut\Sigmat^{-1}\Utp\tr\! K + \Utp \Sigmat^{-1}\Ut\tr\! K) Z P_{\Vt}^{\perp}.
   \end{split}
\end{equation}

We now consider the Riemannian Hessian parameterization (i.e., the Hessian represented by the triplet $ (\Mh, \Uhp, \Vhp) $)
\begin{equation}\label{eq:B.5}
	\Hessian f(X)[H]=\Ut\Mh\Vt\tr + \Uhp\Vt\tr + \Ut\Vhp\tr,
\end{equation}
and exploit the usual technique of multiplying by the left and right by appropriate matrices, together with the modified orthogonality conditions and gauge conditions, in order to explicitly compute the parameters $ \Mh $, $ \Uhp $ and $ \Vhp $.

To find $ \Mh $, we start by left-multiplying~\eqref{eq:B.5} by $ \Ut\tr \! K\tr $ and right-multiplying by $ \Xi_0\Vt $, obtaining
\begin{equation}\label{eq:B.6}
   \Mh = \Ut\tr \! K\tr (\Hessian f(X)[H]) \, \Xi_0\Vt.
\end{equation}
Using the fact that the Riemannian Hessian is the orthogonal projection of~\eqref{eq:derivativesmoothextensionGradient} onto the tangent space at $X$, and using the definition of $ \PTXM(\cdot) $, we have
\begin{equation}\label{eq:B.7}
\begin{split}
   \Hessian f(X)[H] &= \PTXM(\D\!\bar{G}(X)[H]) \\ 
                    &= P_{\Ut} (\D\!\bar{G}(X)[H]) \, P_{\Vt} + P_{\Ut}^{\perp} (\D\!\bar{G}(X)[H]) \, P_{\Vt} + P_{\Ut} (\D\!\bar{G}(X)[H]) \, P_{\Vt}^{\perp}.
\end{split}
\end{equation}
Inserting~\eqref{eq:B.7} into~\eqref{eq:B.6} we get
\begin{align*}
   \Mh &= \Ut\tr \! K\tr (\Hessian f(X)[H]) \, \Xi_0\Vt \\
             &\xeq[]{\eqref{eq:B.6}} \Ut\tr \! K\tr (\D\!\bar{G}(X)[H]) \, \Xi_0\Vt \\
             &\xeq[]{\eqref{eq:B.4}} \Ut\tr \! K\tr \PTXM(\dt{Z}) \, \Xi_0\Vt.
\end{align*}
Finally, using~\eqref{eq:PTXMdtZ}, we obtain
\begin{equation}\label{eq:Mhat_appendix}
   \Mh = \Ut\tr \! K \dt{Z} \Xi_0\Vt.
\end{equation}

For the $ \Uhp $ factor, right-multiplying~\eqref{eq:B.5} by $ \Xi_0\Vt $, we obtain
\begin{align*}
   (\Hessian f(X)[H]) \, \Xi_0\Vt &= \Ut\Mh\Vt\tr \! \Xi_0\Vt + \Uhp\Vt\tr \! \Xi_0\Vt + U\Vhp\tr \! \Xi_0\Vt \\
                            &= \Ut\Mh + \Uhp,
\end{align*}
from which
\begin{equation}\label{eq:Uhp_intermediate}
   \Uhp = (\Hessian f(X)[H]) \, \Xi_0\Vt - \Ut\Mh.
\end{equation}
Developing the first term $ (\Hessian f(X)[H]) \, \Xi_0\Vt $, we obtain
\begin{align*}
   (\Hessian f(X)[H]) \, \Xi_0\Vt &= \left( \PTXM(\D\!\bar{G}(X)[H]) \right) \Xi_0\Vt \\
                       &\xeq[]{\eqref{eq:B.7}} P_{\Ut} (\D\!\bar{G}(X)[H]) \, P_{\Vt}\Xi_0\Vt + P_{\Ut}^{\perp} (\D\!\bar{G}(X)[H]) \, P_{\Vt}\Xi_0\Vt \\
                       & \ \qquad + P_{\Ut} (\D\!\bar{G}(X)[H]) \,\cancel{P_{\Vt}^{\perp} \Xi_0\Vt} \\
                       &= P_{\Ut} (\D\!\bar{G}(X)[H]) \, \Xi_0\Vt + P_{\Ut}^{\perp} (\D\!\bar{G}(X)[H]) \, \Xi_0\Vt + 0 \\
                       &= (\D\!\bar{G}(X)[H]) \, \Xi_0\Vt.
\end{align*}
Hence, substituting the last expression into~\eqref{eq:Uhp_intermediate}, we obtain
\begin{equation}\label{eq:Uhp_intermediate_n2}
   \Uhp = (\D\!\bar{G}(X)[H]) \, \Xi_0\Vt - \Ut\Mh.
\end{equation}
Now, we focus on the first term in~\eqref{eq:Uhp_intermediate_n2}, i.e.,
\begin{align}
   (\D\!\bar{G}(X)[H]) \, \Xi_0V & \xeq[]{\eqref{eq:B.4}} \left( \PTXM(\dt{Z}) + P_{\Ut}^{\perp} Z (\ldots) + (\ldots) Z P_{\Vt}^{\perp} \right) \Xi_0\Vt \nonumber \\
                             & = \PTXM(\dt{Z}) \, \Xi_0\Vt + P_{\Ut}^{\perp} Z (\Xi_0\Vt\Sigmat^{-1}\Vtp\tr + \Xi_0\Vtp \Sigmat^{-1}\Vt\tr) \, \Xi_0\Vt + 0 \nonumber \\
                             & = \PTXM(\dt{Z}) \, \Xi_0\Vt + P_{\Ut}^{\perp} Z \Xi_0\Vtp \Sigmat^{-1}. \label{eq:B.11}
\end{align}
Developing the first term in the last expression:
\begin{equation}\label{eq:B.12}
   \PTXM(\dt{Z}) \, \Xi_0\Vt \xeq[]{\eqref{eq:PTXMdtZ}} P_{\Ut}\dt{Z}P_{\Vt} \Xi_0\Vt + P_{\Ut}^{\perp}\dt{Z}P_{\Vt} \Xi_0\Vt + \cancel{P_{\Ut}\dt{Z}P_{\Vt}^{\perp} \Xi_0\Vt} = P_{\Ut}\dt{Z}\Xi_0\Vt + P_{\Ut}^{\perp}\dt{Z} \Xi_0\Vt.
\end{equation}
Then, continuing from~\eqref{eq:B.11},
\begin{align*}
   (\D\!\bar{G}(X)[H]) \, \Xi_0\Vt &= \PTXM(\dt{Z}) \, \Xi_0\Vt + P_{\Ut}^{\perp} Z \Xi_0\Vtp \Sigmat^{-1} \\
                             & \xeq[]{\eqref{eq:B.12}} P_{\Ut}\dt{Z}\Xi_0\Vt + P_{\Ut}^{\perp}\dt{Z}\Xi_0\Vt + P_{\Ut}^{\perp} Z \Xi_0 \Vtp \Sigmat^{-1} \\
                             & = P_{\Ut}\dt{Z}\Xi_0\Vt + P_{\Ut}^{\perp} \! \left(\dt{Z}\Xi_0\Vt + Z \Xi_0 \Vtp \Sigmat^{-1}\right).
\end{align*}
Inserting the last expression into~\eqref{eq:Uhp_intermediate_n2}, and using \eqref{eq:Mhat_appendix}, we get
\begin{align*}
   \Uhp &= P_{\Ut}\dt{Z}\Xi_0\Vt + P_{\Ut}^{\perp} \! \left(\dt{Z}\Xi_0\Vt + Z \Xi_0 \Vtp \Sigmat^{-1}\right) - \Ut\Mh \\
                          & \xeq[]{\eqref{eq:Mhat_appendix}} P_{\Ut}\dt{Z}\Xi_0\Vt + P_{\Ut}^{\perp} \! \left(\dt{Z}\Xi_0\Vt + Z \Xi_0 \Vtp \Sigmat^{-1}\right) - \Ut \Ut\tr \! K \dt{Z} \Xi_0\Vt \\
                          &= \cancel{P_{\Ut}\dt{Z}\Xi_0\Vt} + P_{\Ut}^{\perp} \! \left(\dt{Z}\Xi_0\Vt + Z \Xi_0 \Vtp \Sigmat^{-1}\right) \cancel{- \Ut \Ut\tr \! K \dt{Z} \Xi_0\Vt}. \\
                          & = P_{\Ut}^{\perp} \! \left(\dt{Z}\Xi_0\Vt + Z \Xi_0 \Vtp \Sigmat^{-1}\right).
\end{align*}
So finally
\[
   \Uhp = P_{\Ut}^{\perp} \! \left(\dt{Z}\Xi_0\Vt + Z \Xi_0 \Vtp \Sigmat^{-1}\right).
\]

To compute the factor $ \Vhp $, we first transpose~\eqref{eq:B.5}, obtaining
\[
   (\Hessian f(X)[H])\tr = \Vt \Mh\tr\! \Ut\tr + 
   \Vt \Uhp + \Vhp \Ut\tr.
\]
Right-multiplying by $ K\Ut $:
\begin{align*}
   (\Hessian f(X)[H])\tr \! K\Ut & = \Vt \Mh\tr\! \Ut\tr \! K\Ut + \Vt \cancel{\Uhp K\Ut} + \Vhp \Ut\tr \! K\Ut \\
                            & = \Vt \Mh\tr + \Vhp
\end{align*}
Therefore, $\Vhp$ is given by
\begin{equation}\label{eq:Vhp_intermediate}
   \Vhp = (\Hessian f(X)[H])\tr K\Ut - \Vt \Mh\tr.
\end{equation}
Now, we focus on the first term in~\eqref{eq:Vhp_intermediate}, namely, $ (\Hessian f(X)[H])\tr K\Ut $. We transpose~\eqref{eq:B.7}, obtaining
\[
   (\Hessian f(X)[H])\tr = P_{\Vt}\tr (\D\!\bar{G}(X)[H])\tr P_{\Ut}\tr + P_{\Vt}\tr (\D\!\bar{G}(X)[H])\tr (P_{\Ut}^{\perp})\tr + (P_{\Vt}^{\perp})\tr (\D\!\bar{G}(X)[H])\tr P_{\Ut}\tr.
\]
Substituting this expression and the transpose of~\eqref{eq:Mhat_appendix} into~\eqref{eq:Vhp_intermediate}, we get
\begin{align}
   \Vhp & = P_{\Vt}\tr (\D\!\bar{G}(X)[H])\tr P_{\Ut}\tr K\Ut + \cancel{P_{\Vt}\tr (\D\!\bar{G}(X)[H])\tr (P_{\Ut}^{\perp})\tr \! K\Ut} \nonumber \\
                          & \quad + (P_{\Vt}^{\perp})\tr (\D\!\bar{G}(X)[H])\tr P_{\Ut}\tr K\Ut - \Vt \Mh\tr \nonumber \\
                          & = P_{\Vt}\tr (\D\!\bar{G}(X)[H])\tr \! K\Ut + (P_{\Vt}^{\perp})\tr (\D\!\bar{G}(X)[H])\tr \! K\Ut - \Vt \Mh\tr \nonumber \\
                          & \xeq[]{\eqref{eq:Mhat_appendix}} (\D\!\bar{G}(X)[H])\tr \! K\Ut - \Vt \Vt\tr\! \Xi_0 \dt{Z}\tr \! K\Ut  \nonumber \\
                          & = (\D\!\bar{G}(X)[H])\tr \! K\Ut - P_{\Vt}\tr \dt{Z}\tr \! K\Ut.  \label{eq:Vhp_intermediate_n2}
\end{align}
Focus on the first term $ (\D\!\bar{G}(X)[H])\tr \! K\Ut $:
\begin{align*}
   (\D\!\bar{G}(X)[H])\tr & \xeq[]{\eqref{eq:B.4}} \left( \PTXM(\dt{Z}) + P_{\Ut}^{\perp} Z (\ldots) + (\ldots) Z P_{\Vt}^{\perp} \right)\tr \\
                          & = (\PTXM(\dt{Z}))\tr + (\ldots)\tr Z\tr (P_{\Ut}^{\perp})\tr + (P_{\Vt}^{\perp})\tr Z\tr (\ldots)\tr \\
                          & \xeq[]{\eqref{eq:PTXMdtZ}} P_{\Vt}\tr \dt{Z}\tr P_{\Ut}\tr + P_{\Vt}\tr \dt{Z}\tr (P_{\Ut}^{\perp})\tr + (P_{\Vt}^{\perp})\tr \dt{Z}\tr P_{\Ut}\tr \\
                          & \ \qquad + (\ldots)\tr Z\tr (P_{\Ut}^{\perp})\tr + (P_{\Vt}^{\perp})\tr Z\tr (\ldots)\tr .
\end{align*}
Right-multiply by $ K\Ut $:
\begin{align*}
   (\D\!\bar{G}(X)[H])\tr \! K\Ut & = P_{\Vt}\tr \dt{Z}\tr P_{\Ut}\tr K\Ut + P_{\Vt}\tr \dt{Z}\tr (P_{\Ut}^{\perp})\tr K\Ut + (P_{\Vt}^{\perp})\tr \dt{Z}\tr P_{\Ut}\tr K\Ut \\
                                & \quad + (\ldots)\tr Z\tr (P_{\Ut}^{\perp})\tr K\Ut + (P_{\Vt}^{\perp})\tr Z\tr (\ldots)\tr K\Ut  \\
                                & = P_{\Vt}\tr \dt{Z}\tr P_{\Ut}\tr K\Ut + (P_{\Vt}^{\perp})\tr \dt{Z}\tr K\Ut \\
                                & \quad + (P_{\Vt}^{\perp})\tr Z\tr (K U_{\mathrm{p}} \Sigmat^{-1} \Ut\tr + K\Ut\Sigmat^{-1} \Utp\tr) K\Ut  \\
                                & = P_{\Vt}\tr \dt{Z}\tr \! K\Ut + (P_{\Vt}^{\perp})\tr \dt{Z}\tr \! K\Ut + (P_{\Vt}^{\perp})\tr Z\tr K \Utp \Sigmat^{-1} \\
                                & = P_{\Vt}\tr \dt{Z}\tr \! K\Ut + (P_{\Vt}^{\perp})\tr ( \dt{Z}\tr \! K\Ut + Z\tr \! K \Utp \Sigmat^{-1}).
\end{align*}
Finally, inserting the last expression into~\eqref{eq:Vhp_intermediate_n2}, we obtain
\begin{align*}
   \Vhp & = (\D\!\bar{G}(X)[H])\tr \! K\Ut - P_{\Vt}\tr \dt{Z}\tr \! K\Ut \\
                          & = P_{\Vt}\tr \dt{Z}\tr \! K\Ut + (P_{\Vt}^{\perp})\tr ( \dt{Z}\tr \! K\Ut + Z\tr \! K \Utp \Sigmat^{-1}) - P_{\Vt}\tr \dt{Z}\tr \! K\Ut \\
                          & = (P_{\Vt}^{\perp})\tr ( \dt{Z}\tr \! K\Ut + Z\tr \! K \Utp \Sigmat^{-1}).
\end{align*}

\bigskip

Summarizing, the parameterization of the Riemannian Hessian $ \Hessian f(X)[H] $
is given by $ (\Mh, \Uhp, \Vhp) $, with
\[
	\Mh = \Ut\tr \! K \dt{Z} \Xi_0\Vt, \qquad \Uhp = P_{\Ut}^{\perp} \! \left(\dt{Z}\Xi_0\Vt + Z \Xi_0 \Vtp \Sigmat^{-1}\right),
\]
\[
	\Vhp = (P_{\Vt}^{\perp})\tr ( \dt{Z}\tr \! K\Ut + Z\tr \! K \Utp \Sigmat^{-1}),
\]
concluding the proof of \eqref{eq:hatM}--\eqref{eq:hatVp}.

\section{Calculation of the directional derivative of the truncation of the weighted SVD and of its adjoint} \label{SM:derivative_svd_truncation}
In this appendix, we detail the calculations to compute the directional derivative of the (weigthed) rank-$\rt$ trunctation, and of its adjoint with respect to the standard Frobenious inner product.
Let
\[
	Y = \sum_{i = 1}^{r} \sigma_{i} u_{i} v_{i}\tr,
\]
with the modified orthogonality conditions (according to the preconditioned inner product)
\[
	u_{\ell}\tr K u_{j} = \delta_{\ell j}, \qquad v_{\ell}\tr \Xi_{0} v_{j} = \delta_{\ell j},\quad \ell,j\in\{1,\dots,r\}.
\]
The rank-$\rt$ truncation of $Y$ is
\[
	\cT_{\rt}(Y) = \sum_{i=1}^{\rt} \sigma_{i} u_{i} v_{i}\tr,
\]
and we denote the projectors related to the first $\rt$ singular vectors by
\[
	P_{U_{Z}} = U_{Z}U_{Z}\tr\! K, \qquad P_{V_{Z}} = \Xi_{0} V_{Z}V_{Z}\tr\! ,
\]
and those associated to the first $r$ singular vectors by
\[
	P_{U_{Y}} = U_{Y}U_{Y}\tr\! K, \qquad P_{V_{Y}} = \Xi_{0} V_{Y}V_{Y}\tr\! .
\]
We further split $U_Y$ and $V_Y$ into the two terms $U_Y=\left[U_Z\; U_{Z_\perp}\right]$ and $V_Y=\left[V_Z\; V_{Z_\perp}\right]$.
Our derivation consists in three steps.

\paragraph{Step 1} We perturb $Y$ along a direction $H\in\mathbb{R}^{m\times n}$, 
\[
	Y(t) = Y + t H,
\]
we denote with $\sigma_{\ell}(t)$, $u_{\ell}(t)$ and $v_{\ell}(t)$, the singular values and singular vectors of $Y(t)$, and recall the relations
\begin{equation*}
	\begin{cases}
		Y(t) \Xi_{0} v_{\ell}(t) = \sigma_{\ell}(t) u_{\ell}(t), \\
		Y(t)\tr\! K u_{\ell}(t) = \sigma_{\ell}(t) v_{\ell}(t).
	\end{cases}
\end{equation*}

Deriving these expressions and evaluating them at $t=0$ (omitting the dependence on $t$ for brevity), we obtain
\begin{equation*}
	\begin{cases}
		\dt{Y} \Xi_{0} v_{\ell} + Y \Xi_{0} \dt{v}_{\ell} = \dt{\sigma}_{\ell} u_{\ell} + \sigma_{\ell} \dt{u}_{\ell}, \\
		\dt{Y}\tr\! K u_{\ell} + Y\tr\! K \dt{u}_{\ell} = \dt{\sigma}_{\ell} v_{\ell} + \sigma_{\ell} \dt{v}_{\ell},
	\end{cases}
\end{equation*}
which, replacing $ \dt{Y} = H $, becomes 
\begin{equation*}
	\begin{cases}
		H \Xi_{0} v_{\ell} + Y \Xi_{0} \dt{v}_{\ell} = \dt{\sigma}_{\ell} u_{\ell} + \sigma_{\ell} \dt{u}_{\ell}, \qquad \text{(A)} \\  
		H\tr\! K u_{\ell} + Y\tr\! K \dt{u}_{\ell} = \dt{\sigma}_{\ell} v_{\ell} + \sigma_{\ell} \dt{v}_{\ell} \qquad \text{(B)}.
	\end{cases}
\end{equation*}
Next, we take the inner product of (A) with $ u_{\ell}\tr K $,
\[
	u_{\ell}\tr \! K H \Xi_{0} v_{\ell} + \underbrace{u_{\ell}\tr \! K Y \Xi_{0} \dt{v}_{\ell}}_{= \sigma_{\ell}v_{\ell}\tr \! \Xi_{0}\dt{v}_{\ell} = 0} = \dt{\sigma}_{\ell} \underbrace{u_{\ell}\tr \! K u_{\ell}}_{=1} + \sigma_{\ell} \underbrace{u_{\ell}\tr \! K \dt{u}_{\ell}}_{=0},
\]
where we used the fact that the derivative of $\| u_{\ell}(t) \|^{2} = 1$ with respect to $t$ implies $u_{\ell}\tr \! K \dt{u}_{\ell} = 0$.
We obtain an expression for the derivative at $t=0$ of the singular values, namely 
\[
	 \dt{\sigma}_{\ell} = u_{\ell}\tr \! K H \Xi_{0} v_{\ell}.
\]
\paragraph{Step 2} We now consider the expansions
\begin{equation} \label{eq:expansions}
	\dt{u}_{\ell} = \sum_{j \neq \ell} a_{\ell j} u_{j}, \qquad \dt{v}_{\ell} = \sum_{j \neq \ell} b_{\ell j} v_{j},
\end{equation}
where the index $\ell$ is excluded from the summation since, evaluating at $t=0$, the derivative of the relation $u_{\ell}(t)\tr K u_{\ell}(t)=1$ leads to $u_{\ell}\tr K \dt{u}_\ell=0$.
Still using the orthogonality conditions, it follows that $ u_{j}\tr \! K \dt{u}_{\ell} = a_{\ell j} $. Now, we take the inner product of (A) with $ u_{j}\tr \! K $, obtaining
\[
	u_{j}\tr \! K H \Xi_{0} v_{\ell} + \underbrace{u_{j}\tr \! K Y \Xi_{0} \dt{v}_{\ell}}_{= \sigma_{j}v_{j}\tr \! \Xi_{0} \dt{v}_{\ell} = \sigma_{j} b_{\ell j}} = \dt{\sigma}_{\ell} \underbrace{u_{j}\tr \! K u_{\ell}}_{=0} + \underbrace{\sigma_{\ell} u_{j}\tr \! K \dt{u}_{\ell}}_{= \sigma_{\ell} a_{\ell j}},
\]
\[
	u_{j}\tr \! K H \Xi_{0} v_{\ell} + \sigma_{j} b_{\ell j} = \sigma_{\ell} a_{\ell j}.
\]
Taking the inner product of (B) with $ v_{j}\tr \! \Xi_{0} $ yields
\[
	v_{j}\tr \! \Xi_{0} H\tr\! K u_{\ell} + \underbrace{v_{j}\tr \! \Xi_{0} Y\tr\! K \dt{u}_{\ell}}_{= \sigma_{j} u_{j}\tr \! K \dt{u}_{\ell} = \sigma_{j} a_{\ell j}} = \dt{\sigma}_{\ell} \underbrace{v_{j}\tr \! \Xi_{0} v_{\ell}}_{= 0} + \underbrace{\sigma_{\ell} v_{j}\tr \! \Xi_{0} \dt{v}_{\ell}}_{= \sigma_{\ell} b_{\ell j}},
\]
\[
	v_{j}\tr \! \Xi_{0} H\tr\! K u_{\ell} + \sigma_{j} a_{\ell j} = \sigma_{\ell} b_{\ell j}.
\]
and we can deduce the linear system of two equations in the two unknowns $ a_{\ell j} $ and $ b_{\ell j} $:
\begin{equation*}
	\begin{cases}
		u_{j}\tr \! K H \Xi_{0} v_{\ell} + \sigma_{j} b_{\ell j} = \sigma_{\ell} a_{\ell j}, \\
		u_{\ell}\tr \! K H \Xi_{0} v_{j} + \sigma_{j} a_{\ell j} = \sigma_{\ell} b_{\ell j}.
	\end{cases}
\end{equation*}
Letting $ x \coloneqq u_{j}\tr \! K H \Xi_{0} v_{\ell} $, and $ y \coloneqq v_{j}\tr \! \Xi_{0} H\tr\! K u_{\ell} $, the solutions are
\begin{equation*}
	\begin{cases}
		a_{\ell j} = \displaystyle\frac{\sigma_{\ell} x + \sigma_{j} y}{\sigma_{\ell}^{2} - \sigma_{j}^{2}}, \\
		b_{\ell j} = \displaystyle\frac{\sigma_{\ell} y + \sigma_{j} x}{\sigma_{\ell}^{2} - \sigma_{j}^{2}}.
	\end{cases}
\end{equation*}
This holds as long as $ \sigma_{\ell}^{2} - \sigma_{j}^{2} \neq 0 $.

\paragraph{Step 3}
The derivative of the rank-$\rt$ truncation of $Y$ in the direction $H$ is
\[
	\D \cT_{\rt}(Y)[H] = \sum_{i=1}^{\rt} \dt{\sigma}_{i} u_{i}v_{i}\tr + \sum_{i=1}^{\rt} \sigma_{i} \dt{u}_{i} v_{i}\tr + \sum_{i=1}^{\rt} \sigma_{i} u_{i} \dt{v}_{i}\tr.
\]
Inserting the expansions~\eqref{eq:expansions}, we get
\[
	\D \cT_{\rt}(Y)[H] = \underbrace{\sum_{i=1}^{\rt} \dt{\sigma}_{i} u_{i}v_{i}\tr}_{\eqqcolon T_{1}} + \underbrace{\sum_{i=1}^{\rt} \sum_{j \neq i} \sigma_{i} a_{ij} u_{j} v_{i}\tr}_{\eqqcolon T_{2}} + \underbrace{\sum_{i=1}^{\rt} \sum_{j \neq i} \sigma_{i} b_{ij} u_{i} v_{j}\tr}_{\eqqcolon T_{3}}.
\]
Now we split the contributions of $T_{2}$ and $T_{3}$ according to $ 1 \leq j \leq \rt $, $ \rt < j \leq r $, and $ j \geq r $.

\paragraph{Contributions for $ 1 \leq j \leq \rt $} We define $ M_{ij} \coloneqq u_{i}\tr \! K H \Xi_{0} v_{j} $, so that $ x = M_{ji} $ and $ y = M_{ij} $. Then
\[
	\sigma_{i} a_{ij} = \frac{\sigma_{i}^{2} M_{ji} + \sigma_{i}\sigma_{j} M_{ij}}{\sigma_{i}^{2} - \sigma_{j}^{2}}, \qquad \sigma_{i} b_{ij} = \frac{\sigma_{i}^{2} M_{ij} + \sigma_{i}\sigma_{j} M_{ji}}{\sigma_{i}^{2} - \sigma_{j}^{2}}.
\]
The outer product $ u_{\alpha} v_{\beta}\tr $ (with $ \alpha = i $, $ \beta = j $, $ i, j \leq \rt $) gets from $T_{2}$ the contribution ($j = \alpha$, $i = \beta$, $\alpha \neq \beta$)  and from $T_{3}$ the contribution ($i = \alpha$, $j = \beta$, $\alpha \neq \beta$).

Summing the two contributions (from $T_{2}$ + from $T_{3}$)
\begin{align*}
	\text{from } T_{2} + \text{from } T_{3} &= \frac{\sigma_{\beta}^{2} M_{\alpha\beta} + \sigma_{\beta}\sigma_{\alpha} M_{\beta\alpha}}{\sigma_{\beta}^{2} - \sigma_{\alpha}^{2}} + \frac{\sigma_{\alpha}^{2} M_{\alpha\beta} + \sigma_{\alpha}\sigma_{\beta} M_{\beta\alpha}}{\sigma_{\alpha}^{2} - \sigma_{\beta}^{2}} \\
	&= \frac{-\sigma_{\beta}^{2} M_{\alpha\beta} - \sigma_{\beta}\sigma_{\alpha} M_{\beta\alpha} + \sigma_{\alpha}^{2} M_{\alpha\beta} + \sigma_{\alpha}\sigma_{\beta} M_{\beta\alpha}}{\sigma_{\alpha}^{2} - \sigma_{\beta}^{2}} \\
	&= M_{\alpha\beta}.
\end{align*}
Putting all together ($T_{1}$ with $\alpha = \beta$, $T_{2}$ with $\alpha \neq \beta$, $\alpha, \beta \in \{ 1, \rt \} $, $T_{3}$ with $\alpha \neq \beta$, $\alpha, \beta \in \{ 1, \rt \} $), we obtain
\[
	\sum_{\alpha, \beta \leq \rt} M_{\alpha\beta} u_{\alpha}v_{\beta}\tr = U_{Z}U_{Z}\tr K H \Xi_{0} V_{Z}V_{Z}\tr = P_{U_{Z}}\tr H P_{V_{Z}}.
\]

\paragraph{Contributions with $ j > r \implies \sigma_{j} = 0 $} We get
\[
	\sigma_{i} a_{ij} = \frac{\sigma_{i}^{2} M_{ji}}{\sigma_{i}^{2}} = M_{ji}, \qquad \sigma_{i} b_{ij} = \frac{\sigma_{i}^{2} M_{ij}}{\sigma_{i}^{2}} = M_{ij}.
\]
Summing the two contributions
\[
	\sum_{i=1}^{\rt} \sum_{j > r} \sigma_{i} a_{ij} u_{j} v_{i}\tr + \sum_{i=1}^{\rt} \sum_{j > r} \sigma_{i} b_{ij} u_{i} v_{j}\tr = \underbrace{\sum_{i=1}^{\rt} \sum_{j > r} M_{ji} u_{j} v_{i}\tr}_{T_{2}} + \underbrace{\sum_{i=1}^{\rt} \sum_{j > r} M_{ij} u_{i} v_{j}\tr}_{T_{3}}.
\]
Now we observe that from $T_{2}$
\[
	\sum_{i=1}^{\rt} \sum_{j > r} (u_{j}\tr \! K H \Xi_{0} v_{i}) u_{j} v_{i}\tr = \sum_{i=1}^{\rt} \sum_{j > r} (u_{j}u_{j}\tr) K H \Xi_{0} v_{i}v_{i}\tr = (I - P_{U_{Y}}) H P_{V_{Z}},
\]
while from $T_{3}$
\[
	\sum_{i=1}^{\rt} \sum_{j > r} (u_{i}u_{i}\tr) K H \Xi_{0} v_{j}v_{j}\tr = P_{U_{Z}} H (I - P_{V_{Y}}).
\]

\paragraph{Contributions $\rt < j \leq r$, $\sigma_{j} \neq 0$}
We define
\[
	\Xi_{ij} \coloneqq \frac{\sigma_{i}^{2}}{\sigma_{i}^{2} - \sigma_{j}^{2}} = 1 + \frac{\sigma_{j}^{2}}{\sigma_{i}^{2} - \sigma_{j}^{2}}, \qquad \Psi_{ij} \coloneqq \frac{\sigma_{i} \sigma_{j}}{\sigma_{i}^{2} - \sigma_{j}^{2}}.
\]
and obtain the expressions
\begin{equation} \label{eq:184}
	\sigma_{i} a_{ij} = \frac{\sigma_{i}^{2} M_{ji} + \sigma_{i}\sigma_{j} M_{ij}}{\sigma_{i}^{2} - \sigma_{j}^{2}} = \Xi_{ij} M_{ji} + \Psi_{ij} M_{ij},
\end{equation}
\[
	\sigma_{i} b_{ij} = \frac{\sigma_{i}^{2} M_{ij} + \sigma_{i}\sigma_{j} M_{ji}}{\sigma_{i}^{2} - \sigma_{j}^{2}} = \Xi_{ij} M_{ij} + \Psi_{ij} M_{ji}.
\]
We introduce the blocks
\[
	H_{21} = U_{Z_{\perp}}\tr \! K H \Xi_{0}	V_{Z},\qquad H_{12} = U_{Z}\tr K H \Xi_{0} V_{Z_{\perp}}.
\]
Note that $ (H_{21})_{ji} = M_{ji} = u_{j}\tr \! K H \Xi_{0} v_{i} $. Then, inserting \eqref{eq:184}, 
\begin{align*}
	\sum_{i=1}^{\rt}\sum_{j=\rt+1}^{r} \sigma_{i} a_{ij} u_{j} v_{i}\tr &= \sum_{i=1}^{\rt}\sum_{j=\rt+1}^{r} (\Xi_{ij} M_{ji} + \Psi_{ij} M_{ij}) u_{j} v_{i}\tr \\
	&= U_{Z_{\perp}} \left[ \Xi \odot H_{21} + \Psi \odot H_{12}\tr \right] V_{Z}\tr \\
	&= U_{Z_{\perp}} H_{21} V_{Z}\tr + U_{Z_{\perp}} \! \left[ (\Xi-I) \odot H_{21} + \Psi \odot H_{12}\tr \right] V_{Z}\tr.
\end{align*}
Similarly,
\begin{align*}
	\sum_{i=1}^{\rt}\sum_{j=\rt+1}^{r} \sigma_{i} b_{ij} u_{i} v_{j}\tr &= U_{Z} \left[ \Xi\tr\! \odot H_{12} + \Psi\tr\! \odot H_{21}\tr \right] V_{Z_{\perp}}\tr \\
	&= U_{Z} H_{12} V_{Z_{\perp}}\tr + U_{Z} \left[ (\Xi\tr\! - I) \odot H_{12} + \Psi\tr\! \odot H_{21}\tr \right] V_{Z_{\perp}}\tr.
\end{align*}
Collecting all terms, we get the final expression of the directional derivative,
\begin{equation}
	\begin{split}
		\D \cT_{r}(Y)[H] = & \ P_{U_{Z}} H P_{V_{Z}} + (I - P_{U_{Y}}) H P_{V_{Z}} + P_{U_{Z}} H (I - P_{V_{Y}}) \\
		& \ + U_{Z_{\perp}} H_{21} V_{Z}\tr + U_{Z_{\perp}} \! \left[ (\Xi-I) \odot H_{21} + \Psi \odot H_{12}\tr \right] V_{Z}\tr \\
		& \ + U_{Z} H_{12} V_{Z_{\perp}}\tr + U_{Z} \left[ (\Xi\tr\! - I) \odot H_{12} + \Psi\tr\! \odot H_{21}\tr \right] V_{Z_{\perp}}\tr.
	\end{split}
\end{equation}
An alternative expression is obtained by manipulating a few terms as
\begin{equation*}
\begin{aligned}
&P_{U_{Z}} H P_{V_{Z}} + (I - P_{U_{Y}}) H P_{V_{Z}} + P_{U_{Z}} H (I - P_{V_{Y}})+ U_{Z_\perp} H_{21}V_{Z}\tr + U_Z H_{12} V_{Z_\perp}\tr\\
\;&=P_{U_{Z}} H P_{V_{Z}} + (I - P_{U_{Y}}) H P_{V_{Z}} + P_{U_{Z}} H (I - P_{V_{Y}})\\
& + P_{U_{Z_\perp}} H P_{V_{Z}} + P_{U_Z} H P_{V_{Z_\perp}}\\
&=P_{U_Z} H + H P_{V_Z} - P_{U_Z}H P_{V_Z}.
\end{aligned}
\end{equation*}
A compact expression is then
\begin{equation}
	\begin{split}
		\D \cT_{\rt}(Y)[H] = & P_{U_Z} H + H P_{V_Z} - P_{U_Z}H P_{V_Z} \\
		& \ + U_{Z_{\perp}} \! \left[ (\Xi-I) \odot H_{21} + \Psi \odot H_{12}\tr \right] V_{Z}\tr \\
		& \ + U_{Z} \left[ (\Xi\tr\! - I) \odot H_{12} + \Psi\tr\! \odot H_{21}\tr \right] V_{Z_{\perp}}\tr,
	\end{split}
\end{equation}
with 
\[
	H_{21} = U_{Z_{\perp}}\tr \! K H \Xi_{0}	V_{Z},\qquad H_{12} = U_{Z}\tr K H \Xi_{0} V_{Z_{\perp}}.
\]

\paragraph{Calculation of the adjoint operator of $\D \cT_{r}(Y)[H]$}
\[
	\langle \D \cT_{\rt}(Y)[H], \varOmega \rangle = \langle H, (\D \cT_{\rt}(Y))^{\ast}[\varOmega] \rangle.
\]
We calculate the adjoint term by term:
\[
	\langle U_{Z}U_{Z}\tr KH, \varOmega \rangle = \langle H, K U_{Z}U_{Z}\tr \varOmega \rangle.
\]
\[
	\langle H \Xi_{0} V_{Z}V_{Z}\tr, \varOmega \rangle = \langle H, \varOmega V_{Z}V_{Z}\tr \Xi_{0} \rangle.
\]
\[
	\langle - U_{Z}U_{Z}\tr KH \Xi_{0} V_{Z}V_{Z}\tr, \varOmega \rangle = -\langle H, K U_{Z}U_{Z}\tr \varOmega V_{Z}V_{Z}\tr \Xi_{0} \rangle.
\]
The part of $ U_{Z_{\perp}} \! \left[ (\Xi-I) \odot H_{21} + \Psi \odot H_{12}\tr \right] V_{Z}\tr $ relative to $(\Xi-I)$:
\begin{align*}
	\left\langle U_{Z_{\perp}} \! \left[ (\Xi-I) \odot H_{21} \right] V_{Z}\tr, \varOmega \right\rangle &= \langle U_{Z_{\perp}} \! \left[ (\Xi-I) \odot U_{Z_{\perp}}\tr \! KH \Xi_{0} V_{Z} \right] V_{Z}\tr, \varOmega \rangle \\
	&= \langle \left[ (\Xi-I) \odot U_{Z_{\perp}}\tr \! KH \Xi_{0} V_{Z} \right], U_{Z_{\perp}}\tr \! \varOmega V_{Z}\rangle \\
	&= \langle U_{Z_{\perp}}\tr \! KH \Xi_{0} V_{Z}, (\Xi-I) \odot (U_{Z_{\perp}}\tr \! \varOmega V_{Z})\rangle \\
	&= \langle H, K U_{Z_{\perp}} \! \big[(\Xi-I) \odot \underbrace{(U_{Z_{\perp}}\tr \! \varOmega V_{Z})}_{\eqqcolon \varOmega_{21}}\big] V_{Z}\tr \! \Xi_{0} \rangle \\
	&= \langle H, K U_{Z_{\perp}} \! \left[(\Xi-I) \odot \varOmega_{21} \right] V_{Z}\tr \! \Xi_{0} \rangle.   \tag{I.a}
\end{align*}
The part of $ U_{Z_{\perp}} \! \left[ (\Xi-I) \odot H_{21} + \Psi \odot H_{12}\tr \right] V_{Z}\tr $ relative to $\Psi$:
\begin{align*}
	\left\langle U_{Z_{\perp}} \! \left[ \Psi \odot H_{12}\tr \right] V_{Z}\tr, \varOmega \right\rangle &= \left\langle U_{Z_{\perp}} \! \left[ \Psi \odot (U_{Z}\tr K H \Xi_{0} V_{Z_{\perp}})\tr \right] V_{Z}\tr, \varOmega \right\rangle \\
	&= \left\langle U_{Z_{\perp}} \! \left[ \Psi \odot (V_{Z_{\perp}}\tr \! \Xi_{0} H\tr \! K U_{Z}) \right] V_{Z}\tr, \varOmega \right\rangle \\
	&= \left\langle  \Psi \odot (V_{Z_{\perp}}\tr \! \Xi_{0} H\tr \! K U_{Z}), U_{Z_{\perp}}\tr \! \varOmega V_{Z} \right\rangle \\
	&= \left\langle V_{Z_{\perp}}\tr \! \Xi_{0} H\tr \! K U_{Z}, \Psi \odot (U_{Z_{\perp}}\tr \! \varOmega V_{Z}) \right\rangle \\
	&= \langle H\tr, \Xi_{0} V_{Z_{\perp}} [ \Psi \odot \underbrace{(U_{Z_{\perp}}\tr \! \varOmega V_{Z})}_{\eqqcolon \varOmega_{21}} ] U_{Z}\tr \! K \rangle \\
	&= \left\langle H, KU_{Z} \left[ \Psi\tr\! \odot \varOmega_{21}\tr \right] V_{Z_{\perp}}\tr \Xi_{0} \right\rangle.   \tag{I.b}
\end{align*}
The last term $U_{Z} \left[ (\Xi\tr\! - I) \odot H_{12} + \Psi\tr\! \odot H_{21}\tr \right] V_{Z_{\perp}}\tr$ yields, for the part relative to $(\Xi\tr\! - I)$,
\begin{align*}
	\left\langle U_{Z} \left[ (\Xi\tr\! - I) \odot H_{12} \right] V_{Z_{\perp}}\tr, \varOmega \right\rangle &= \langle (\Xi\tr\! - I) \odot H_{12}, \underbrace{U_{Z}\tr \varOmega  V_{Z_{\perp}}}_{\eqqcolon \varOmega_{12}} \rangle \\
	&= \left\langle (\Xi\tr\! - I) \odot (U_{Z}\tr K H \Xi_{0} V_{Z_{\perp}}), \varOmega_{12} \right\rangle \\
	&= \left\langle U_{Z}\tr K H \Xi_{0} V_{Z_{\perp}}, (\Xi\tr\! - I) \odot \varOmega_{12} \right\rangle \\
	&= \left\langle H, K U_{Z} \left[ (\Xi\tr\! - I) \odot \varOmega_{12} \right] V_{Z_{\perp}}\tr \Xi_{0} \right\rangle.   \tag{II.a}
\end{align*}
For the part of $U_{Z} \left[ (\Xi\tr\! - I) \odot H_{12} + \Psi\tr\! \odot H_{21}\tr \right] V_{Z_{\perp}}\tr$ relative to $\Psi$, we obtain
\begin{align*}
	\left\langle U_{Z} \left[ \Psi\tr\! \odot H_{21}\tr \right] V_{Z_{\perp}}\tr, \varOmega \right\rangle &= \langle \Psi\tr\! \odot H_{21}\tr, \underbrace{U_{Z}\tr \varOmega  V_{Z_{\perp}}}_{\eqqcolon \varOmega_{12}} \rangle \\
	&= \langle \Psi\tr\! \odot (U_{Z_{\perp}}\tr \! K H \Xi_{0} V_{Z})\tr, \varOmega_{12} \rangle \\
	&= \langle V_{Z}\tr \Xi_{0} H\tr \! K U_{Z_{\perp}}, \Psi\tr\! \odot \varOmega_{12} \rangle \\
	&= \langle H\tr, \Xi_{0} V_{Z} (\Psi\tr\! \odot \varOmega_{12}) U_{Z_{\perp}}\tr\! K  \rangle \\
	&= \langle H, K U_{Z_{\perp}} (\Psi \odot \varOmega_{12}\tr) V_{Z}\tr\! \Xi_{0} \rangle.   \tag{II.b}
\end{align*}
Collecting all terms,
\[
	\langle \D \cT_{\rt}(Y)[H], \varOmega \rangle = \langle H, (\D \cT_{\rt}(Y))^{\ast}[\varOmega] \rangle,
\]
\begin{equation*}
	\begin{split}
		(\D \cT_{\rt}(Y))^{\ast}[\varOmega] =& \ K U_{Z}U_{Z}\tr \varOmega + \varOmega V_{Z}V_{Z}\tr \Xi_{0} - K U_{Z}U_{Z}\tr \varOmega V_{Z}V_{Z}\tr \Xi_{0} \\
		& \ + K U_{Z_{\perp}} \! \left[(\Xi-I) \odot \varOmega_{21} \right] V_{Z}\tr \! \Xi_{0} + KU_{Z} \left[ \Psi\tr\! \odot \varOmega_{21}\tr \right] V_{Z_{\perp}}\tr \Xi_{0} \\
		& \ + K U_{Z} \left[ (\Xi\tr\! - I) \odot \varOmega_{12} \right] V_{Z_{\perp}}\tr \Xi_{0} + K U_{Z_{\perp}} (\Psi \odot \varOmega_{12}\tr) V_{Z}\tr\! \Xi_{0},
	\end{split}
\end{equation*}
\begin{equation*}
	\begin{split}
		(\D \cT_{\rt}(Y))^{\ast}[\varOmega] =& \ K U_{Z}U_{Z}\tr \varOmega + \varOmega V_{Z}V_{Z}\tr \Xi_{0} - K U_{Z}U_{Z}\tr \varOmega V_{Z}V_{Z}\tr \Xi_{0} \\
		& \ + K U_{Z} \left[ (\Xi\tr\! - I) \odot \varOmega_{12} \right] V_{Z_{\perp}}\tr \Xi_{0} + KU_{Z} \left[ \Psi\tr\! \odot \varOmega_{21}\tr \right] V_{Z_{\perp}}\tr \Xi_{0} \\
		& \ + K U_{Z_{\perp}} \! \left[(\Xi-I) \odot \varOmega_{21} \right] V_{Z}\tr \! \Xi_{0} + K U_{Z_{\perp}} (\Psi \odot \varOmega_{12}\tr) V_{Z}\tr\! \Xi_{0},
	\end{split}
\end{equation*}
\begin{equation} \label{eq:adjoint_deriv}
	\begin{split}
		(\D \cT_{\rt}(Y))^{\ast}[\varOmega] =& \ K U_{Z}U_{Z}\tr \varOmega + \varOmega V_{Z}V_{Z}\tr \Xi_{0} - K U_{Z}U_{Z}\tr \varOmega V_{Z}V_{Z}\tr \Xi_{0} \\
		& \ + K U_{Z} \left[ \Psi\tr\! \odot \varOmega_{21}\tr + (\Xi\tr\! - I) \odot \varOmega_{12} \right] V_{Z_{\perp}}\tr \Xi_{0} \\
		& \ + K U_{Z_{\perp}} \! \left[ \Psi \odot \varOmega_{12}\tr + (\Xi-I) \odot \varOmega_{21} \right] V_{Z}\tr \! \Xi_{0},
	\end{split}
\end{equation}
with $ \varOmega_{12} = U_{Z}\tr \varOmega  V_{Z_{\perp}} $ and $ \varOmega_{21} =  U_{Z_{\perp}}\tr\! \varOmega V_{Z} $.

\bibliographystyle{aomalpha}

\begin{footnotesize}
   \bibliography{Lowrank_parametrized_systems.bib}
\end{footnotesize}

\end{document}